\documentclass[12pt]{article}
\usepackage{geometry}
\usepackage{tipa}
\usepackage{latexsym,amssymb,amsfonts}
\usepackage{mathrsfs}
\usepackage{graphicx,color}
\usepackage{psfrag,hyperref}
\usepackage{amsmath, amsbsy}
\usepackage{amsopn, amstext}
\usepackage{amsmath,multicol,amsopn}
\usepackage{latexsym,amssymb}
\usepackage{amsbsy, amstext,makeidx}
\usepackage{algorithmicx,algorithm}
\usepackage{epstopdf}
\usepackage{epsfig}
\usepackage{times}
\usepackage{multirow}

\usepackage{bm}
\usepackage{enumerate}

\usepackage{bbm}
\usepackage{cite}
\usepackage{enumerate}
\usepackage{textcomp}
\usepackage{rotating}
\usepackage{adjustbox}

\def\cB{{\cal B}}

\def\cT{{\cal T}} 
\def\cH{{\cal H}}

\def\va{{\bf a}}

\def\vc{{\bf c}}
\def\vd{{\bf d}}
\def\ve{{\bf e}}

\def\vg{{\bf g}}

\def\vk{{\bf k}}
\def\vl{{\bf l}}

\def\vs{{\bf s}}

\def\vu{{\bf u}}
\def\vv{{\bf v}}
\def\vw{{\bf w}}
\def\vx{{\bf x}}
\def\vy{{\bf y}}
\def\vz{{\bf z}}

\def\mZero{{\bf 0}}

\def\vzero{{\bf 0}}

\def\vtheta{{\bm\theta}}

\def\vxi{{\bm\xi}}

\def\square#1{\vbox{\hrule\hbox{\vrule height#1%
     \kern#1\vrule}\hrule}}
\def\rectangle#1#2{\vbox{\hrule\hbox{\vrule height#1%
     \kern#2\vrule}\hrule}}

\font\tenbb=msbm10 \font\sevenbb=msbm7 \font\fivebb=msbm5

\newfam\bbfam
\scriptscriptfont\bbfam=\fivebb \textfont\bbfam=\tenbb
\scriptfont\bbfam=\sevenbb

\newtheorem{remark}{Remark}

\newtheorem{theorem}{Theorem}

\newtheorem{definition}{Definition}
\newtheorem{proposition}{Proposition}

\def\BlackBox{{\rule{1.5ex}{1.5ex}}} 
\newenvironment{proof}{\par\noindent{\bf Proof\ }}{\hfill\BlackBox\\[2mm]}

\makeatletter
   
   \@addtoreset{equation}{section}
\makeatother

\begin{document}

\title{A sparse semismooth Newton based augmented Lagrangian method for large-scale support vector machines}

\author{\small Dunbiao Niu\thanks{College of Mathematics, Sichuan University, No.24 South Section 1, Yihuan Road, Chengdu 610065,  China. ({\tt dunbiaoniu\_sc@163.com}).}~
        ~Chengjing Wang\thanks{Corresponding author. School of Mathematics, Southwest Jiaotong University, No.999, Xian Road, West Park, High-tech Zone, Chengdu 611756, China. ({\tt renascencewang@hotmail.com}).}~
        ~Peipei Tang\thanks{School of Computing Science, Zhejiang University City College, Hangzhou 310015, China. ({\tt tangpp@zucc.edu.cn}).}~
        ~Qingsong Wang\thanks{School of Mathematics, Southwest Jiaotong University, No.999, Xian Road, West Park, High-tech Zone, Chengdu 611756, China. ({\tt nothing2wang@hotmail.com}).}~
        ~and Enbin Song\thanks{College of Mathematics, Sichuan University, No.24 South Section 1, Yihuan Road, Chengdu 610065,  China. ({\tt e.b.song@163.com}).}
}
\date{}
\maketitle

\begin{abstract}
Support vector machines (SVMs) are successful modeling and prediction tools with a variety of applications. Previous work has demonstrated the superiority of the SVMs in dealing with the high dimensional, low sample size problems. However, the numerical difficulties of the SVMs will become severe with the increase of the sample size. Although there exist many solvers for the SVMs, only few of them are designed by exploiting the special structures of the SVMs. In this paper, we propose a highly efficient sparse semismooth Newton based augmented Lagrangian method for solving a large-scale convex quadratic programming problem with a linear equality constraint and a simple box constraint, which is generated from the dual problems of the SVMs. By leveraging the primal-dual error bound result, the fast local convergence rate of the augmented Lagrangian method can be guaranteed. Furthermore, by exploiting the second-order sparsity of the problem when using the semismooth Newton method, the algorithm can efficiently solve the aforementioned difficult problems. Finally, numerical comparisons demonstrate that the proposed algorithm outperforms the current state-of-the-art solvers for the large-scale SVMs.
\end{abstract}

 \textbf{Keywords}: support vector machines, semismooth Newton method, augmented Lagrangian method

\section{Introduction}
Support Vector Machines (SVMs), introduced by \cite{Vapnik1963Pattern}, are originally formulated for binary classification problems that aim to separate two data sets with the widest margin. Nowadays, the SVMs have been extended to solve a variety of pattern recognition and data mining problems such as feature selection \cite{Taira1999Feature}, text categorization \cite{Goudjil2018A}, hand-written character recognition \cite{Azim2016}, image classification \cite{Lin2011Large}, and so on. Among the different applications of the SVMs, we first introduce two specific examples.
\begin{itemize}
   \item The C-Support Vector Classification  (C-SVC) \cite{Cortes1995Support,Boser1992A}

   Given a training set of instance-label pairs $(\tilde{\vx}_{i},y_{i}),\,i = 1, . . . , n$, where $\tilde{\vx}_{i} \in \mathbb{R}^{q}$ and $y_{i} \in \{+1, -1\}$, the C-SVC aims to find a hyperplane in a given reproducing kernel Hilbert space $\cH$ \cite{aronszajn1950theory} to separate the data set into two classes with the widest margin. In this model, we need to solve the following optimization problem
\begin{equation}\label{Def_C_SVM_Problem}
\begin{split}
\min_{w,b,\xi}\hspace*{2mm} &\frac{1}{2}\langle\vw,\vw\rangle_{\cH}+C \sum_{i=1}^{n}{\xi_{i}}\\
\text{s.t.}\hspace*{2mm} &y_{i}\left(\langle\vw,\phi(\tilde{\vx}_{i})\rangle_{\cH}+b\right)\geq1-\xi_{i},\\
&  \xi_{i} \geq 0,\ i=1,\ldots,n,\\
\end{split}
\end{equation}
where $C > 0$ is a regularization parameter, $\langle\cdot,\cdot\rangle_{\cH}$ is the inner product in a reproducing kernel Hilbert space $\cH$ and $\phi: \mathbb{R}^{q}\rightarrow \cH$ is a feature map such that the function $K(\tilde{\vx}_{i},\tilde{\vx}_{j}):=\langle\phi(\tilde{\vx}_{i}),\phi(\tilde{\vx}_{j})\rangle_{\cH}$ is a reproducing kernel of $\cH$ for any data  $\tilde{\vx}_{i}, \tilde{\vx}_{j} \in \mathbb{R}^{q}$. For example, $K(\tilde{\vx}_{i},\tilde{\vx}_{j})=\tilde{\vx}_{i}^{T}\tilde{\vx}_{j}$ is a linear kernel and $K(\tilde{\vx}_{i},\tilde{\vx}_{j})=\text{exp}^{- || x||^2/2\alpha}$ is a radial basis function (RBF) kernel, where $\alpha > 0$ is a fixed parameter called the width. As it has been shown in \cite{Cortes1995Support}, the dual of the problem
\eqref{Def_C_SVM_Problem} is the following quadratic programming problem
\begin{equation}\label{Def_Dual_of_C_SVM_Problem}
\begin{split}
\min_{\vx\in \mathcal{R}^n}\hspace*{2mm} &\frac{1}{2}\vx^{T}Q\vx-\ve^{T}\vx\\
\text{s.t.}\hspace*{2mm} &\vy^{T}\vx=0,\\
& 0\leq x_{i} \leq C,\ i=1,\ldots,n,\\
\end{split}
\end{equation}
where $\ve = [1,\ldots,1]^T \in \mathbb{R}^{n}$ is a vector of all ones, $\vy=[y_1,\ldots,y_n]^{T}\in \mathbb{R}^{n}$ is a label vector and $Q\in \mathbb{S}^{n}$ (the space of $n\times n$ symmetric matrices) is a positive semidefinite matrix with $Q_{ij}=y_i y_j K(\tilde{\vx}_{i},\tilde{\vx}_{j}),\,i,j=1,\ldots,n$.

\item The Support Vector Regression (SVR) \cite{Vapnik1998SVR}

  For a set of training points $(\tilde{\vx}_{i},y_{i})\ ,i = 1, . . . , n$, the regression problem is to predict the quantitative response $y_{i}\in \mathbb{R}$ on the basis of the predictor variable $\tilde{\vx}_{i}\in \mathbb{R}^q$. Given a regularization coefficient $C>0$ and an insensitivity parameter $\epsilon>0$, the standard form of the support vector regression is
\begin{equation}\label{Def_SVR_Problem}
\begin{split}
\min_{\vw,b,\vxi,\vxi^{*}}\hspace*{2mm} &\frac{1}{2}\langle\vw,\vw\rangle_{\cH}+C\sum_{i=1}^{n}{\xi_{i}}+C\sum_{i=1}^{n}{\xi^{*}_{i}}\\
\text{s.t.}\hspace*{2mm} & \langle\vw,\phi(\tilde{\vx}_{i})\rangle_{\cH}+b-y_{i}\geq \epsilon+\xi_{i},\\
&y_{i} - \langle\vw,\phi(\tilde{\vx}_{i})\rangle_{\cH}-b\geq \epsilon-\xi^{*}_{i},\\
&  \xi_{i},\xi^{*}_{i} \geq 0,\ i=1,\ldots,n.
\end{split}
\end{equation}
Furthermore, the dual of the problem \eqref{Def_SVR_Problem} is
\begin{equation}\label{Def_Dual_of_SVR_Problem}
\begin{split}
\min_{\vx,\vz\in \mathcal{R}^n}\hspace*{2mm} &\frac{1}{2}[\vx^{T},\vz^{T}]\left(
 \begin{array}{cc}
K & -K \\
-K & K \\
\end{array}
\right)
\left[
  \begin{array}{c}
    \vx \\
    \vz \\
  \end{array}
\right]+\sum_{i=1}^{n}{(\varepsilon+y_{i})x_{i}}+\sum_{i=1}^{n}{(\varepsilon-y_{i})z_{i}}
\\
\text{s.t.}\hspace*{2mm} &[\ve^{T},-\ve^{T}]
\left[
  \begin{array}{c}
    \vx \\
    \vz \\
  \end{array}
\right]=0,\\
& 0\leq x_{i},z_{i} \leq C,\ i=1,\ldots,n,\\
\end{split}
\end{equation}
where $K\in \mathbb{S}^{n}_{+}$ is a positive semidefinite kernal matrix with $K_{ij}=K(\tilde{\vx}_{i},\tilde{\vx}_{j}),\,i,j=1,\ldots,n$.

\end{itemize}

The above two problems \eqref{Def_C_SVM_Problem} and \eqref{Def_SVR_Problem} show that there are various formulations for the SVMs in different scenarios. However, their dual problems \eqref{Def_Dual_of_C_SVM_Problem} and \eqref{Def_Dual_of_SVR_Problem} can be summarized as the following unified form

\begin{equation}\label{Problem_P1_classical1}
\begin{split}
\min_{\vx\in \mathcal{R}^n}\hspace*{2mm} &\frac{1}{2} \vx^{T}Q\vx+\vc^{T}\vx\\
\text{s.t.}\hspace*{2mm} &\va^{T}\vx=d,\\
& \vl \leq \vx \leq \vu,
\end{split}
\end{equation}
where $Q\in \mathbb{S}^{n}$ is a positive semidefinite matrix and $\vc, \va, \vl, \vu \in \mathbb{R}^{n}$ and $d\in \mathbb{R}$ are given vectors and scalar, respectively. Moreover, we assume that $\vl < \vu$, i.e., $l_{i} < u_{i}$ for all $i\in\{1,\ldots,n\}$, where $l_{i}$ and $u_{i}$ are the $i$th elements of $\vl$ and $\vu$, respectively.
Currently most of the work on computational aspects of the SVMs concentrate on solving the dual problem \eqref{Problem_P1_classical1} since it is a more general framework to handle the SVMs. But there still exist some algorithms that solve the primal problem. For example, Yin and Li \cite{Yin2019A} proposed a semismooth Newton method to solve the primal problems of L2-loss SVC model and $\epsilon$-L2-loss SVR model with linear kernels recently.

    For the convex problem \eqref{Problem_P1_classical1}, many existing optimization algorithms including the accelerated proximal gradient (APG) method, the alternating direction method of multipliers (ADMM) and the interior-point method (IPM) et al., can be applied to solve it efficiently when the problem scale is small or moderate. However,
    when facing the large-scale problems, the numerical difficulties become severe.
    Specifically, when the dimension $n$ is very large, the full storage of the $n \times n$ dense matrix $Q$ in \eqref{Problem_P1_classical1} is very difficult and even impossible for a typical computer.
    Therefore we cannot apply the standard quadratic programming solvers which require the full storage of $Q$ directly.
   An alternative approach is to compute the elements of $Q$ from the original data when it is required. However, this becomes prohibitively time consuming since the elements of $Q$ are frequently required in the process of solving the problem \eqref{Problem_P1_classical1}.

  Currently, one approach to avoid using the whole elements of the matrix $Q$ is to adopt the decomposition strategy \cite{Osuna1997An,Joachims1998Making,Platt1999Fast}. In this kind of approach, only a small subset of variables in each iteration need to be updated so that only a few rows of the matrix $Q$ are involved in, which significantly reduces the computational cost.
    The Sequential Minimal Optimization (SMO) method \cite{Platt1999Fast} is one of the well-known decomposition methods, in which only two variables are considered in each iteration.
     The popular SVMs solver LIBSVM \cite{Chang2011LIBSVM} also implements an SMO-type method \cite{Fan2005Working} to solve the problem \eqref{Problem_P1_classical1}.

  Another widely used algorithm for the problem \eqref{Problem_P1_classical1} is the gradient projection (GP) method.
  Thanks to the low-cost algorithm of the projection onto the feasible set of \eqref{Problem_P1_classical1} and the identification properties of the GP method \cite{Calamai1987Projected},
  we only need to consider a reduced subproblem which is related to the current active set of variables.
  Furthermore, inspired by the algorithm for the bound constrained quadratic programming problem \cite{Mor1991On}, the proportionality-based $2$-phase gradient projection (P2GP) method \cite{Serafino2018A} is derived to solve the problem \eqref{Problem_P1_classical1}. In addition, the fast APG (FAPG) method \cite{ito2017unified} is another commonly used algorithm to solve various classification models, including the C-SVC, L2-loss SVC, and $\textit{v}$-SVM et al..

   The advantage of the decomposition and PG methods is that only a small subset of variables are involved in each subproblem, i.e., only a small part of the elements of the matrix $Q$ are required in each iteration.
   However, as the numerical experiments shown in Section \ref{section_4}, both algorithms may encounter the problem of slow convergence. If $Q$ is positive definite and the nondegeneracy assumption holds, the SMO-type decomposition methods are shown to be linear convergent to the solution of \eqref{Problem_P1_classical1} in \cite{Chen2006A}. It is well-known that the PG-type methods can exhibit (linear) sublinear convergence when the objective function is (strongly) convex. These existing theories partly explain why the convergences of these algorithms are not ideal.

 In this paper, we aim to solve the problem \eqref{Problem_P1_classical1} by applying the augmented Lagrangian (AL) method \cite{Rockafellar1976Augmented} to the dual problem of \eqref{Problem_P1_classical1}. Meanwhile, a semismooth Newton (SsN) method is used to solve the inner subproblems of the AL method. It is well known that to fully fulfill the potential of the AL method, the inner subproblems should be solved accurately and efficiently. Although the SsN method is an ideal approach to solve the inner subproblems, it is not proper to be applied directly because the cost of the SsN method may be very high at the beginning few iterations due to the lack of sparse structure of the generalized Hessian, especially for very large-scale problems. To overcome this difficulty, we may use some algorithm, say, the random Fourier features method \cite{rahimi2008random} which will be introduced in Section \ref{section_4}, to produce a good initial point, and then transfer to the semismooth Newton based augmented Lagrangian (SsNAL) method. The generalized Hessian for the inner subproblem may probably has some kind of sparse structure. Wisely exploiting this nice structure may largely reduce the computational cost and the memory consumption in each SsN step.
Hence, our proposed algorithm not only has the same advantage as the decomposition method in \cite{Joachims1998Making} with memory requirements being linear to the number of training examples and support vectors, but also has the fast local convergent rate in both inner and outer iterations. Since our algorithm fully takes advantage of the sparse structure, we call it a Sparse SsN based AL (SSsNAL) method. Besides, there are three main reasons why the SSsNAL method can be implemented efficiently to solve the problem \eqref{Problem_P1_classical1}:
 \begin{itemize}
   \item[(\uppercase\expandafter{\romannumeral1})] The piecewise linear-quadratic structure of the problem \eqref{Problem_P1_classical1} guarantees the fast local convergence of the AL method \cite{Robinson1981Some,Sun1986monotropic}.
   \item[(\uppercase\expandafter{\romannumeral2})] There exist many efficient algorithms \cite{Helgason1980A,BRUCKER1984163,Cosares1994STRONGLY,Kiwiel2008Variable,ito2017unified} to compute the value of the Euclidean projection of any given point onto the feasible set \eqref{Problem_P1_classical1} due to its special structure. Furthermore, the explicit formula of the generalized Jacobian, which is named HS-Jacobian \cite{Han1997Newton}, can be easily derived.
   \item[(\uppercase\expandafter{\romannumeral3})] It is generally true for many SVMs that the number of the support vectors are much less than that of the training examples, and many multiplier variables of the support vectors are at the upper bound of the box constraint \cite{Joachims1998Making}. That is, very few of the components of the optimal solution lie in the interior of the box constraint.
 \end{itemize}

As will be shown later, the above (\uppercase\expandafter{\romannumeral1}) and (\uppercase\expandafter{\romannumeral2}) guarantee the inner subproblem can be solved by the SsN method with a very low memory consumption in each iteration. And the above (\uppercase\expandafter{\romannumeral1}), (\uppercase\expandafter{\romannumeral2}) and (\uppercase\expandafter{\romannumeral3}) together provide an insight into the compelling advantages of applying the SSsNAL method to the problem \eqref{Problem_P1_classical1}. Indeed, for the large-scale problem \eqref{Problem_P1_classical1}, the numerical experiments in Section \ref{section_4} will show that the SSsNAL method only needs at most a few dozens of outer iterations to reach the desired solutions while all the inner subproblems can be solved without too much effort.

The rest of the paper is organized as follows. In Section \ref{Preliminaries}, some preliminaries about the restricted dual formulation and some necessary error bound results are provided. Section \ref{section_3} is dedicated to present the SSsNAL method for the restricted dual problem in details. Numerical experiments are presented on real data in Section \ref{section_4}, which verify the performance of our SSsNAL method against other solvers. Finally, Section \ref{section_5} concludes this paper.

\textbf{Notations:} Let $\mathcal{X}$ and $\mathcal{Y}$ be two real finite dimensional Euclidean spaces. For any convex function $p: S\subset\mathcal{X} \rightarrow (-\infty,\infty]$, its conjugate function is denoted by $p^*$, i.e., $p^*(x)=\sup_{y}\{\langle x,y \rangle-p(y)\}$, and its subdifferential at $x$ is denoted by $\partial p(x)$, i.e., $\partial p(x):=\{y\,|\,p(z)\geq p(x)+\langle y,z-x\rangle, \forall\ z\in \text{dom}(p)\}$. For a given closed convex set $\Omega$ and a vector $x$, we denote the distance from $x$ to $\Omega$ by $\textrm{dist}(x,\Omega) := \inf_{y\in \Omega}{||x-y||}$ and the Euclidean projection of $x$ onto $\Omega$ by $\Pi_{\Omega}(x) := \arg\min_{y\in \Omega}{||x-y||}$. For any set-valued mapping $F: \mathcal{X} \rightrightarrows \mathcal{Y}$, we use $\text{gph}\ F$ to denote the graph of $F$, i.e., $\text{gph}\ F:=\{(x,y)\in \mathcal{X}\times \mathcal{Y}\ |\ y \in F(x) \}$. We use $I_{n}$ to denote the $n\times n$ identity matrix in $\mathbb{R}^n$ and $A^\dag$ to denote the Moore-Penrose pseudo-inverse of a given matrix $A \in \mathbb{R}^{n\times n}$.

\section{Preliminaries}\label{Preliminaries}
In this section, we present some necessary error bound results, which will be used in the convergence rate analysis of the AL method in Section \ref{subsec:SSsNAL}.

We denote the single linear constraint and the box constraint in the problem \eqref{Problem_P1_classical1} by
\begin{equation}\label{Def L and K}
\begin{split}
L:=\{\vx \in \mathbb{R}^{n}\ | \ \va^{T}\vx=d \}\textrm{ and }K:=\{\vx \in \mathbb{R}^{n}\ | \ \vl\leq \vx \leq \vu \},
\end{split}
\end{equation}
respectively. Then the problem \eqref{Problem_P1_classical1} can be equivalently rewritten as
\begin{eqnarray*}\label{Problem_Primal}
(\textbf{P}) & \min\limits_{\vx\in \mathbb{R}^n} \left\{f(\vx):=\frac{1}{2}\langle \vx, Q\vx\rangle+\langle \vc,\vx\rangle+\delta_{\text{\tiny $K\bigcap L$ }}(\vx)\right\},
\end{eqnarray*}
where $\delta_{\text{\tiny $K\bigcap L$ }}$ is the indicator function for the polyhedral convex set $K\cap L$, i.e., $\delta_{\text{\tiny $K\bigcap L$ }}(\vx)=0$ if $\vx \in K\cap L$, otherwise $\delta_{\text{\tiny $K\bigcap L$ }}(\vx)=+\infty$. Note that the problem (\textbf{P}) is already the dual formulation of the SVMs in the introduction, but we still regard it as the primal problem based on our custom. The dual of the problem (\textbf{P}) is

\begin{eqnarray*}\label{Problem_P1_dual}
(\textbf{D}) & \min\limits_{\vw\in \mathbb{R}^n,\ \vz\in \mathbb{R}^n} \left\{\frac{1}{2}\langle \vw, Q\vw\rangle+\delta^{*}_{\text{\tiny $K\bigcap L$ }}(\vz)\,|\, Q\vw+\vz+\vc=0,\ \vw \in \text{Ran}(Q)\right\},
\end{eqnarray*}
where $\delta^{*}_{\text{\tiny $K\bigcap L$ }}$ is the conjugate of the indicator function $\delta_{\text{\tiny $K\bigcap L$ }}$, and Ran$(Q)$ denotes the range space of $Q$. Note that the additional constraint $\vw \in \text{Ran}(Q)$ is reasonable because, for any $\vw^{0} \in \text{Ran}^{\bot}(Q)$, $\vw$ and $\vw+\vw^{0}$ have the same objective function values and both satisfy the linear constraint in (\textbf{D}). As will be shown in the next section, the constraint $\vw\in \text{Ran}(Q)$ in fact plays an important role to guarantee that the subproblem has a unique solution and our designed algorithm is efficient. Since we restrict $\vw$ in the range space artificially, we may also call (\textbf{D}) a restricted dual problem. Correspondingly, the Karush-Kuhn-Tucker (KKT) condition associated with the problem (\textbf{P}) is given by
\begin{equation}\label{KKT_(DUAL)}
\begin{split}
\vx-\text{Prox}_{\delta_{\text{\tiny $K\bigcap L$ }}}(\vx+\vz)=0,\ Q\vw-Q\vx=0,\ Q\vw+\vz+\vc=0,
\end{split}
\end{equation}
where the proximal mapping for a given closed proper convex function $p: \mathbb{R}^{n}\to (-\infty,+\infty]$ is defined by
$$\text{Prox}_{p}(\vu):=\arg\min_{\vx}\big{\{}p(\vx)+\frac{1}{2}||\vu-\vx||^2\big{\}}, \forall\, \vu\in \mathbb{R}^n.$$
Moreover, for any given parameter $\lambda > 0$, we introduce the following Moreau identity, which will be used frequently.
\begin{equation}\label{Moreau-identity}
\begin{split}
\text{Prox}_{\lambda p}(\vu)+\lambda \text{Prox}_{ p^*/\lambda}(\vu/\lambda)=\vu.
\end{split}
\end{equation}

Let $l$ be the ordinary Lagrangian function for the problem (\textbf{D})
\begin{equation}\label{Def Lagrangian fun for dual}
l(\vw,\vz,\vx)=
\begin{cases}
\frac{1}{2}\langle \vw, Q\vw\rangle+\delta_{\text{\tiny $K\bigcap L$ }}^*(\vz)-\langle \vx, Q\vw+\vz+\vc \rangle, & \ \vw \in \text{Ran}(Q), \\
+\infty,  &\ \ \text{otherwise}.
\end{cases}
\end{equation}
Then, we define the maximal monotone operators $\cT_{f}$ and $\cT_{l}$  \cite{Rockafellar1976Augmented} by
\begin{equation*}\label{Def_T_f}
\begin{split}
\cT_{f}(\vx):=&\partial f(\vx)=\{\vu_{x}\in \mathbb{R}^{n}|\ \vu_{x} \in   Q\vx+\vc+\partial \delta_{\text{\tiny $K\bigcap L$ }}(\vx)\},\ \forall\, \vx \in \mathbb{R}^{n}\\
\end{split}
\end{equation*}
and
\begin{equation*}\label{Def_T_l}
\begin{split}
\cT_{l}(\vw,\vz,\vx):=&\{(\vu_{w},\vu_{z},\vu_{x})\in \mathbb{R}^{3n}\ |\ (\vu_{w},\vu_{z},-\vu_{x})\in \partial l(\vw,\vz,\vx)\},
\end{split}
\end{equation*}
respectively.
Correspondingly, the inverses of $\cT_{f}$ and $\cT_{l}$ are given, respectively, by
\begin{equation*}\label{set-valued map inverse Tl}
\begin{split}
\cT_{f}^{-1}(\vu_{x}):=&\{\vx\in \mathbb{R}^{n}\ |\ \vu_{x}\in \partial f(\vx)\},\\
\end{split}
\end{equation*}
and
\begin{equation*}\label{set-valued map inverse Tl}
\begin{split}
\cT_{l}^{-1}(\vu_{w},\vu_{z},\vu_{x}):=&\{(\vw,\vz,\vx)\in \mathbb{R}^{3n}\ |\ (\vu_{w},\vu_{z},-\vu_{x})\in \partial l(\vw,\vz,\vx)\}.
\end{split}
\end{equation*}
Recall that a closed proper convex function $g: \mathcal{X}\rightarrow (-\infty,+\infty]$ is said to be \textit{piecewise linear-quadratic} if dom $g$ is the union of finitely many polyhedral sets and on each of these polyhedral sets, $g$ is either an affine or a quadratic function \cite[Definition 10.20]{Rockafellar1998Variational}. Hence the objective function $f$ in (\textbf{P}) is piecewise linear-quadratic. Meanwhile, by \cite[Theorem 11.14]{Rockafellar1998Variational} the support function $\delta_{\text{\tiny $K\bigcap L$ }}^*$ is piecewise linear-quadratic, which implies that $l$ is also piecewise linear-quadratic.

In addition, $F: \mathcal{X} \rightrightarrows \mathcal{Y}$ is called \textit{piecewise polyhedral} if its graph is the union of finitely many polyhedral convex sets. Therefore, according to the following proposition established in \cite{Sun1986monotropic}, both $\cT_{f}(\vx)$ and $\cT_{l}(\vw,\vz,\vx)$ are piecewise polyhedral multivalued mappings.
\begin{proposition}\label{prop:piecewise polyhedral}
\cite{Sun1986monotropic} Let $\mathcal{X}$ be a finite-dimensional real Euclidean space and $\theta: \mathcal{X}\rightarrow(-\infty,+\infty]$ be a closed proper convex function. Then $\theta$ is piecewise linear-quadratic if and only if the graph of $\partial\theta$ is piecewise polyhedral. Moreover, $\theta$ is piecewise linear-quadratic if and only if its conjugate $\theta^{*}$ is piecewise
linear-quadratic.
\end{proposition}
In \cite{Robinson1981Some}, Robinson  established the following fundamental property to describe the locally upper Lipschitz continuity of a piecewise polyhedral multivalued mapping.
\begin{proposition}\label{prop locally upper Lipschitz continuity}
\cite{Robinson1981Some} If the multivalued mapping $F:\mathcal{X}\rightrightarrows \mathcal{Y}$ is piecewise polyhedral, then $F$ is locally upper Lipschitz continuous at any $\vx^{0}\in \mathcal{X}$ with modulus $\kappa_{0}$ independent of the choice of $\vx^{0}$. i.e., there exists a neighborhood $V$ of $\vx^{0}$ such that
$F(\vx)\subseteq F(\vx^{0})+\kappa_{0}||\vx-\vx^{0}||\textbf{B}_{\mathcal{Y}},\ \ \forall\, \vx \in V.$
\end{proposition}
Therefore, the above Propositions \ref{prop:piecewise polyhedral} and \ref{prop locally upper Lipschitz continuity} imply that $\cT_{f}(\vx)$ and $\cT_{l}(\vw,\vz,\vx)$ are both locally upper Lipschitz continuous.
Combining \cite[Theorem 3H.3]{Dontchev2013Implicit}, we have the following result.
\begin{proposition}\label{prop locally metrically subregular}
Assume that the KKT system \eqref{KKT_(DUAL)} has at least one solution. Let $(\bar{\vw},\bar{\vz},\bar{\vx})$ be a solution of the KKT system \eqref{KKT_(DUAL)}. Then $\mathcal{T}_{f}^{-1}$ is metrically subregular at $\bar{\vx}$ for the origin and $\mathcal{T}_{l}^{-1}$ is also metrically subregular at $(\bar{\vw},\bar{\vz},\bar{\vx})$ for the origin, i.e., there exist a neighborhood of origin $\mathcal{V}$ and constants $\kappa_{l} >0,\ \kappa_{f} >0$ along with the neighborhoods $\textbf{B}_{\delta_l}(\bar{\vw},\bar{\vz},\bar{\vx})$ and $\textbf{B}_{\delta_f}(\bar{\vx})$ such that
\begin{eqnarray*}\label{metrically subregular Sun}
\text{dist}\big{(}\vx,\mathcal{T}_{f}^{-1}(\mZero)\big{)} \leq  \kappa_{f}\text{dist}\big{(}\mZero,\mathcal{T}_{f}(\vx)\cap\mathcal{V}\big{)},
\end{eqnarray*}
holds for any $\vx \in \textbf{B}_{\delta_f}(\bar{\vx})$ and
\begin{eqnarray*}\label{metrically subregular Sun}
\text{dist}\big{(}(\vw,\vz,\vx),\mathcal{T}_{l}^{-1}(\mZero)\big{)} \leq  \kappa_{l}\text{dist}\big{(}\mZero,\mathcal{T}_{l}(\vw,\vz,\vx)\cap\mathcal{V}\big{)},
\end{eqnarray*}
holds for any $(\vw,\vz,\vx) \in \textbf{B}_{\delta_l}(\bar{\vw},\bar{\vz},\bar{\vx})$.

\end{proposition}
Besides, we may go one step further to present the error bound condition in some semilocal sense. Although Zhang et al. \cite{Zhang2017An} presented a similar result without a proof, for the sake of completeness, we still present the detailed results and the proof.
\begin{proposition}\label{proposition1}
For any $r >0$ and $(\bar{\vw},\bar{\vz},\bar{\vx})\in \mathcal{T}_{l}^{-1}(\mZero)$, there exist $\kappa_{f}(r) >0$ and $\kappa_{l}(r) >0$ such that
\begin{subequations}\label{Def bar_w and bar_v}
\begin{align}
\text{dist}(\vx,\mathcal{T}_{f}^{-1}(\mZero))&\leq \kappa_{f}(r)\ \text{dist}(\mZero,\mathcal{T}_{f}(\vx)), \label{Def sm T_f}\\
\text{dist}\big{(}(\vw,\vz,\vx),\mathcal{T}_{l}^{-1}(\mZero)\big{)}&\leq \kappa_{l}(r)\text{dist}\big{(}\mZero,\mathcal{T}_{l}(\vw,\vz,\vx)\big{)}.\label{Def sm T_l}
\end{align}
\end{subequations}
hold for any $\vx\in \mathcal{R}^{n}$ with $\text{dist}(\vx,\mathcal{T}_{f}^{-1}(\mZero))\leq r$ and $||(\vw,\vz,\vx)-(\bar{\vw},\bar{\vz},\bar{\vx})||\leq r$.
\end{proposition}

\begin{proof}
For the sake of contradiction, we assume that the first assertion  about inequality \eqref{Def sm T_f} is false. Then for some $\tilde{r} >0$ and any $\kappa_{f}(\tilde{r})=k >0$, there exists
$\vx^{k}\in \mathbb{R}^{n}$ with $\text{dist}(\vx^k,\cT_{f}^{-1}(\mZero))\leq \tilde{r}$ such that
\begin{equation}\label{error bound_temp1_contradiction00}
\begin{split}
\text{dist}(\vx^k,\cT_{f}^{-1}(\mZero))> k \text{dist}(\mZero,\mathcal{T}_{f}(\vx^k)).
\end{split}
\end{equation}
Next, by using the fact that $\cT_{f}^{-1}(\mZero)$ is compact, we know that $\{\vx^{k}\}$ is a bounded sequence. Therefore, there exists a subsequence $\{\vx^{k_j}\}$ such that $\vx^{k_j}\rightarrow \vx^{*}$ as $k_j\rightarrow + \infty$. Then, together with \eqref{error bound_temp1_contradiction00}, we have
\begin{equation*}\label{metrically subregular_temp3}
\begin{split}
0\leq \text{dist}\big{(}\mZero,\cT_{f}(\vx^{*})\big{)}=\lim_{k_j\rightarrow + \infty}\text{dist}(\mZero,\mathcal{T}_{f}(\vx^{k_j}))\leq
\lim_{k_j\rightarrow + \infty}\frac{\text{dist}(\vx^{k_j},\cT_{f}^{-1}(\mZero))}{{k_j}}
\leq \lim_{k_j\rightarrow + \infty}\frac{\tilde{r}}{k_j}=0,
\end{split}
\end{equation*}
which implies that $\vx^{*}\in \cT_{f}^{-1}(\mZero)$. Moreover, there exists $\bar{k}>\kappa_{f}$ such that $||\vx^{k_j}-\vx^{*}||\leq \delta_{f}$ holds for all $k_j>\bar{k}$, where the parameters $\kappa_{f}$ and $\delta_{f}$ have been defined in Proposition \ref{prop locally metrically subregular}. Thus, for some $\varepsilon>0$ and all $k_j>\max\{\bar{k},\frac{\tilde{r}}{\varepsilon}\}$, we have
\begin{equation*}\label{error bound temp_2}
\begin{split}
 k_j \text{dist}(\mZero,\mathcal{T}_{f}(\vx^{k_j}))\leq \text{dist}(\vx^{k_j},\cT_{f}^{-1}(\mZero))\leq \kappa_{f}\text{dist}(\mZero,\mathcal{T}_{f}(\vx^{k_j})\cap \cB_{\varepsilon}(\mZero))=\kappa_{f}\text{dist}(\mZero,\mathcal{T}_{f}(\vx^{k_j})),
\end{split}
\end{equation*}
which, together with $k_j>\bar{k}>\kappa_{f}$, is a contradiction. Hence, the first assertion is true.

Similarly, for the sake of contradiction, we also assume that the second assertion about inequality \eqref{Def sm T_l} is false. Then, there exist $\tilde{r}>0$ and $(\tilde{\vw},\tilde{\vz},\tilde{\vx})\in \cT_{l}^{-1}(\mZero)$, for any $\kappa_{l}(\tilde{r})=k >0$, such that
\begin{equation}\label{metrically subregular_temp1_contradiction0}
\begin{split}
\text{dist}\big{(}(\vw^{k},\vz^{k},\vx^{k}),\cT_{l}^{-1}(\mZero)\big{)}> k \text{dist}\big{(}\mZero,\cT_{l}(\vw^{k},\vz^{k},\vx^{k})\big{)},\ \exists\, (\vw^{k},\vz^{k},\vx^{k}) \in \cB_{\tilde{r}}(\tilde{\vw},\tilde{\vz},\tilde{\vx}).
\end{split}
\end{equation}
Note that $\{(\vw^{k},\vz^{k},\vx^{k})\}$ is a bounded sequence. Hence, there exists a subsequence $\{(\vw^{k_i},\vz^{k_i},\vx^{k_i})\}$ such that $(\vw^{k_i},\vz^{k_i},\vx^{k_i})\rightarrow (\tilde{\vw}^{*},\tilde{\vz}^{*},\tilde{\vx}^{*})$ as $k_i\rightarrow + \infty$. Then we have
\begin{equation}\label{metrically subregular_temp2}
\begin{split}
\text{dist}\big{(}\mZero,\cT_{l}(\vw^{k_i},\vz^{k_i},\vx^{k_i})\big{)}
  < \frac{\text{dist}\big{(}(\vw^{k_i},\vz^{k_i},\vx^{k_i}),\cT_{l}^{-1}(\mZero)\big{)}}{k_i}.
\end{split}
\end{equation}
Now taking the limits on both sides of the inequality \eqref{metrically subregular_temp2}, we have
\begin{equation*}\label{metrically subregular_temp3}
\begin{split}
0\leq \text{dist}\big{(}\mZero,\cT_{l}(\tilde{\vw}^{*},\tilde{\vz}^{*},\tilde{\vx}^{*})\big{)}
\leq \lim_{k_i\rightarrow + \infty}\frac{\text{dist}\big{(}(\vw^{k_i},\vz^{k_i},\vx^{k_i}),\cT_{l}^{-1}(\mZero)\big{)}}{k_i}=0,
\end{split}
\end{equation*}
which implies that $(\tilde{\vw}^{*},\tilde{\vz}^{*},\tilde{\vx}^{*})\in \cT_{l}^{-1}(\mZero)$. Note that $\cT_{l}^{-1}$ is metrically subregular at $(\tilde{\vw}^{*},\tilde{\vz}^{*},\tilde{\vx}^{*})$ for the origin. Then, for all $k_i$ sufficiently large and some $\varepsilon>0$, we have
\begin{equation*}\label{metrically subregular_temp3}
\begin{split}
\text{dist}\big{(}(\vw^{k_i},\vz^{k_i},\vx^{k_i}),\cT_{l}^{-1}(\mZero)\big{)}
\leq& \kappa_{l}\text{dist}\big{(}\mZero,\cT_{l}(\vw^{k_i},\vz^{k_i},\vx^{k_i})\cap\cB_{\varepsilon}(\mZero)\big{)}
=\kappa_{l}\text{dist}\big{(}\mZero,\cT_{l}(\vw^{k_i},\vz^{k_i},\vx^{k_i})\big{)}.
\end{split}
\end{equation*}
On the other hand, \eqref{metrically subregular_temp1_contradiction0} implies that
\begin{equation*}\label{metrically subregular_temp4}
\begin{split}
\text{dist}\big{(}(\vw^{k_i},\vz^{k_i},\vx^{k_i}),\cT_{l}^{-1}(\mZero)\big{)}> k_i \text{dist}\big{(}\mZero,\cT_{l}(\vw^{k_i},\vz^{k_i},\vx^{k_i})\big{)}.
\end{split}
\end{equation*}
Thus, by taking $k_i>\text{max}(\varepsilon,\kappa_{l})$, we obtain a contradiction. So \eqref{Def sm T_l} is true.
\end{proof}

\section{The SSsNAL method for the SVM problems}\label{section_3}

In this section, we detailedly discuss how to apply the SSsNAL method to solve the problem (\textbf{D}) and establish its convergence theories.

\subsection{The SSsNAL method for the problem (\textbf{D})}
\label{subsec:SSsNAL}

Firstly, we provide the framework of the SSsNAL method. Given $\sigma>0$, the AL function associated with the problem (\textbf{D}) is given as follows
\begin{eqnarray*}\label{augmented Lagrangian function of D}
\mathcal{L}_{\sigma}(\vw,\vz;\vx)&=&\frac{1}{2}\langle \vw, Q\vw\rangle+\delta^{*}_{\text{\tiny $K\bigcap L$ }}(\vz)+\frac{\sigma}{2}||Q\vw+\vz+\vc-\frac{1}{\sigma}\vx||^{2}-\frac{1}{2\sigma}||\vx||^{2},
\end{eqnarray*}
where $(\vw,\vz,\vx) \in \text{Ran}(Q) \times \mathcal{R}^{n} \times \mathcal{R}^{n} $. The SSsNAL method for solving the problem (\textbf{D}) can be sketched as below.

\begin{algorithm}[H]
\caption{\textbf{:} the SSsNAL method for the problem (\textbf{D})}\label{Def Algorithm ALM}
\begin{algorithmic}
\State Let $\sigma_{0}>0$ be a given  parameter. Choose $\vx^0 \in \mathcal{R}^{n}$. For $k=0,1,2,\ldots,$ generate $(\vw^{k+1},\vz^{k+1})$ and $\vx^{k+1}$ by executing the following iterations:
\State \textbf{Step 1. } Apply the SsN method to compute
\begin{equation}\label{eq inner problem1}
(\vw^{k+1},\vz^{k+1})\approx\operatorname*{argmin}\limits_{\vw \in \text{Ran}(Q), \vz \in \mathbb{R}^{n}}~\big{\{}\Psi^{k}(\vw,\vz):= \mathcal{L}_{\sigma_{k}}(\vw,\vz;\vx^{k})\big{\}}.
\end{equation}
\State \textbf{Step 2. } Compute
\begin{equation*}\label{Def_updata_vxk}
\vx^{k+1}=\vx^{k}-\sigma_{k}(Q\vw^{k+1}+\vz^{k+1}+\vc),
\end{equation*}
and update $\sigma_{k+1}$.
\end{algorithmic}
\end{algorithm}

  Notably, the inner subproblem \eqref{eq inner problem1} has no closed-form solution in general. So we consider how to solve it approximately with the following stopping criteria introduced in \cite{Rockafellar1976Augmented,Rockafellar1976Monotone}:
\begin{equation*}\label{Def_stop_cond_1}
\begin{split}
\text{(A)}&\hspace*{2mm} \Psi^{k}(\vw^{k+1},\vz^{k+1})-\inf_{\vw \in  \text{Ran}(\mathcal{Q}),\vz}{\Psi^{k}(\vw,\vz)}\leq \epsilon_{k}^{2}/(2\sigma_{k}),\,\sum_{k=1}^{+\infty}\epsilon_{k}<+\infty,\\
\text{(B)}&\hspace*{2mm} \Psi^{k}(\vw^{k+1},\vz^{k+1})-\inf_{\vw \in  \text{Ran}(\mathcal{Q}),\vz}{\Psi^{k}(\vw,\vz)}\leq \delta_{k}^{2}/(2\sigma_{k})||\vx^{k+1}-\vx^{k}||^2,\,\,\sum_{k=1}^{+\infty}\delta_{k}<+\infty.
\end{split}
\end{equation*}
Since $\inf_{\vw \in  \text{Ran}(\mathcal{Q}),\vz}\Psi^{k}(\vw,\vz)$ is unknown, in order to apply the stopping criteria (A) and (B) in Algorithm \ref{Def Algorithm ALM}, we need to further analyze the following optimization problem.

\begin{equation*}\label{Def_inner_problem}
\min \big{\{}\Psi^{k}(\vw,\vz)| (\vw,\vz) \in \text{Ran}(Q) \times \mathbb{R}^{n} \big{\}}.
\end{equation*}
Obviously, it is easily seen from the definition of $\Psi^{k}(\vw,\vz)$ in \eqref{eq inner problem1} that the above problem has a unique optimal solution in $\text{Ran}(\mathcal{Q}) \times \mathbb{R}^{n}$. For any $\vw \in \text{Ran}(Q)$, we define
\begin{equation}\label{Def_psi(w)}
\begin{split}
\psi^{k}(\vw):=&\inf_{\vz \in \mathbb{R}^{n}}\Psi^{k}(\vw,\vz)\\
=&\frac{1}{2}\langle \vw, Q\vw\rangle+\delta^{*}_{\text{\tiny $K\bigcap L$ }}(\text{Prox}_{\frac{1}{\sigma_{k}}\delta^{*}_{\text{\tiny $K\bigcap L$ }}}(\vu(\vw)/\sigma_{k}))\\
&+\frac{\sigma_{k}}{2}||\text{Prox}_{\frac{1}{\sigma_{k}}\delta^{*}_{\text{\tiny $K\bigcap L$ }}}(\vu(\vw)/\sigma_{k})-\vu(\vw)/\sigma_{k}||^{2}-\frac{1}{2\sigma_{k}}||\vx^{k}||^{2},\\
=&\frac{1}{2}\langle \vw,Q\vw\rangle+\frac{1}{2\sigma_{k}}\left(||\vu(\vw)||^2-||\vu(\vw)-\Pi_{\text{\tiny $K\bigcap L$ }}(\vu(\vw))||^2\right)-\frac{1}{2\sigma_{k}}||\vx^{k}||^{2},\\
\end{split}
\end{equation}
where $\vu(\vw):=\vx^{k}-\sigma_{k}(Q\vw+\vc)$ and the last equality directly follows from $(2.2)$ in \cite{hiriart1984generalized}. Then $(\vw^{k+1},\vz^{k+1})$ in Step 1 of Algorithm \ref{Def Algorithm ALM} can be obtained in the following manner
\begin{subequations}\label{Def bar_w and bar_v}
\begin{align}
\vw^{k+1}&\approx\arg\min\{\psi^{k}(\vw)|\vw \in \text{Ran}(Q)\},\label{Def bar_w and bar_v1}\\
\vz^{k+1}&=\text{Prox}_{\frac{1}{\sigma_{k}}\delta^{*}_{\text{\tiny $K\bigcap L$ }}}(\vu(\vw^{k+1})/\sigma_{k})=\sigma_{k}^{-1}\big{(}\vu(\vw^{k+1})-\Pi_{\text{\tiny $K\bigcap L$ }}(\vu(\vw^{k+1}))\big{)},\label{Def bar_w and bar_v2}
\end{align}
\end{subequations}
where the last equality in \eqref{Def bar_w and bar_v2} follows from the Moreau identity \eqref{Moreau-identity}. Moreover, in combination with \eqref{Def bar_w and bar_v2}, the update in Step 2 of Algorithm \ref{Def Algorithm ALM} can be simplified as
\begin{equation}\label{Def_new_update_xkplus}
\vx^{k+1}=\vx^{k}-\sigma_{k}\left(Q\vw^{k+1}+\sigma_{k}^{-1}\left(\vu(\vw^{k+1})-\Pi_{\text{\tiny $K\bigcap L$ }}(\vu(\vw^{k+1}))\right)+\vc\right)=\Pi_{\text{\tiny $K\bigcap L$ }}(\vu(\vw^{k+1})).
\end{equation}

 Finally, note that $\psi^{k}(\vw)$ is continuously differentiable and strongly convex with modulus $\tilde{\lambda}_{\min}(Q)$ in $\text{Ran}(Q)$, where $\tilde{\lambda}_{\min}(Q)$ is the minimum nonzero eigenvalue of $Q$. Then we have
\begin{equation}\label{Def_stop_cond_mid}
\begin{split}
&\Psi^{k}(\vw^{k+1},\vz^{k+1})-\inf_{\vw\in  \text{Ran}(\mathcal{Q}),\vz}{\Psi^{k}(\vw,\vz)}\\
&=\psi^{k}(\vw^{k+1})-\inf_{\vw\in  \text{Ran}(\mathcal{Q})}{\psi^{k}(\vw)}
\leq \frac{1}{2\tilde{\lambda}_{\min}(Q)}||\nabla \psi^{k}(\vw^{k+1})||^2,
\end{split}
\end{equation}
where the last inequality is due to Theorem $2.1.10$ in \cite{Nesterov2004Introductory} and the fact that $\nabla \psi^{k}(\vw^{*})=\vzero$ with $\vw^{*}=\arg\min_{\vw\in  \text{Ran}(\mathcal{Q})}{\psi^{k}(\vw)}$. Therefore, we replace the above stopping criteria (A) and (B) by the following easy-to-check criteria
\begin{equation*}\label{Def_stop_cond_2}
\begin{split}
(\text{A}')&\hspace*{2mm} ||\nabla \psi^{k}(\vw^{k+1})|| \leq \sqrt{\tilde{\lambda}_{\min}(Q)} \epsilon_{k}/ \sqrt{\sigma_{k}},\ \epsilon_{k}\geq 0,\ \sum_{k=1}^{+\infty}\epsilon_{k} < +\infty,\\
(\text{B}')&\hspace*{2mm} ||\nabla \psi^{k}(\vw^{k+1})|| \leq \sqrt{\tilde{\lambda}_{\min}(Q)}\delta_{k}/\sqrt{\sigma_{k}}||\vx^{k+1}-\vx^{k}||,\ \delta_{k}\geq 0,\, \sum_{k=1}^{+\infty}\delta_{k}<+\infty.
\end{split}
\end{equation*}

Next, we shall adapt the results developed in \cite{Rockafellar1976Augmented,Rockafellar1976Monotone,luque1984asymptotic} to establish the convergence theory of the AL method for the problem (\textbf{D}). Take a positive scalar $r$ such that $\sum_{k=1}^{+\infty}{\epsilon_{k}} < r$, then we can state the convergence theory as below.
\begin{theorem}\label{theorem4_2}
Let $\{(\vw^k,\vz^k,\vx^k)\}$ be any infinite sequence generated by Algorithm \ref{Def Algorithm ALM} with stopping criterion $(\text{A}')$ and $(\text{B}')$ for solving subproblem \eqref{Def bar_w and bar_v1}. Let $\Omega_{P}$ be the solution set of (\textbf{P}) and $(\vw^*,\vz^*)$ be the unique optimal solution of (\textbf{D}).
Then the sequence $\{\vx^k\}$ converges to $\vx^*\in\Omega_{P}$ and the sequence $(\vw^k,\vz^k)$ converges to the unique optimal solution $(\vw^*,\vz^*)$. Moreover, if $\text{dist}(\vx^{0},\Omega_{P}) \leq r-\sum_{k=1}^{+\infty}{\epsilon_{k}}$, then for all $k>0$ 
\begin{subequations}
\begin{align}
&\text{dist}(\vx^{k+1},\Omega_{P}) \leq \mu_{k}\text{dist}(\vx^{k},\Omega_{P}),\label{Def_sub_eq1}\\
&||(\vw^{k+1},\vz^{k+1})-(\vw^*,\vz^*)||\leq \mu'_{k}||\vx^{k+1}-\vx^{k}||,\label{Def_sub_eq2}
\end{align}
\end{subequations}
where
$\mu_{k}:=\big{[}\delta_{k}+(1+\delta_{k})\kappa_{f}(r)/\sqrt{\kappa^{2}_{f}(r)+\sigma_k^2}\big{]}/(1-\delta_{k})\rightarrow \mu_{\infty}:=\kappa_{f}(r)/\sqrt{\kappa^{2}_{f}(r)+\sigma_{\infty}^2}$,
$\mu'_{k}:=\ \kappa_{l}(r)\sqrt{1/\sigma^{2}_{k}+\tilde{\lambda}_{\min}(Q)\delta^{2}_{k}/\sigma_{k}}\rightarrow
\mu'_{\infty}:=\kappa_{l}(r)/\sigma_{\infty}$,
$\kappa_{f}(r),\ \kappa_{l}(r)>0$ are constant and the parameter $r$ is determined by Proposition \ref{proposition1}. Moreover, $\mu_{k}$ and $\mu'_{k}$ go to $0$ as $\sigma_{k}\uparrow\sigma_{\infty}=+\infty$.
\end{theorem}

\begin{proof}
The statements on the global convergence directly follow from \cite{Rockafellar1976Augmented}.
 The proof for the first inequality \eqref{Def_sub_eq1} follows from the similar idea of the proof in \cite[Lemma 4.1]{Zhang2017An}, so we omit the details. Next, to prove the second inequality \eqref{Def_sub_eq2}, for the given parameter $r$, we have
$$||(\vw^{k+1},\vz^{k+1},\vx^{k+1})-(\vw^*,\vz^*,\vx^{*})||\leq r,\ \forall\, k\geq0,$$
which follows from the fact that $(\vw^{k+1},\vz^{k+1},\vx^{k+1})\rightarrow (\vw^*,\vz^*,\vx^*)$. Therefore, from Proposition \ref{proposition1}, we have
$$||(\vw^{k+1},\vz^{k+1})-(\vw^*,\vz^*)||+\text{dist}(\vx^{k+1},\Omega_{P})\leq \kappa_{l}(r) \text{dist}\big{(}\mZero,\cT_{l}(\vw^{k+1},\vz^{k+1},\vx^{k+1})\big{)},\ \forall\, k\geq0,$$
which, together with the estimate $(4.21)$ in \cite{Rockafellar1976Augmented}, implies
\begin{eqnarray}
||(\vw^{k+1},\vz^{k+1})-(\vw^*,\vz^*)||\leq \kappa_{l}(r) \left[\text{dist}^2(\mZero,\partial \Psi^{k}(\vw^{k+1},\vz^{k+1}))+\sigma_{k}^{-2}||\vx^{k+1}-\vx^{k}||^2\right]^{1/2}.
\label{ineq:dist 0 to Tl}
\end{eqnarray}
Then, according to the update rule of $(\vw^{k+1},\vz^{k+1})$ in \eqref{Def bar_w and bar_v} and the Danskin-type theorem \cite{Danskin1966}, one has
\begin{eqnarray}
\text{dist}(\mZero,\partial \Psi^{k}(\vw^{k+1},\vz^{k+1}))=||\nabla \psi^{k}(\vw^{k+1})||.
\label{eq:dist 0 to partial psi}
\end{eqnarray}
In combination of \eqref{ineq:dist 0 to Tl}, \eqref{eq:dist 0 to partial psi} and the stopping criterion $(B')$, we obtain that for all $k\geq 0$,
$$||(\vw^{k+1},\vz^{k+1})-(\vw^*,\vz^*)||\leq \mu'_{k}||\vx^{k+1}-\vx^{k}||,$$
where $\mu'_{k}:=\ \kappa_{l}(r)\sqrt{1/\sigma^{2}_{k}+\tilde{\lambda}_{\min}(Q)\delta^{2}_{k}/\sigma_{k}}$. This completes the proof of the theorem.
\end{proof}

\subsection{The SsN method for the inner subproblem \eqref{eq inner problem1}}
\label{subsec:SsN}

In this subsection, we propose the SsN method to solve the inner subproblem \eqref{eq inner problem1}. Additionally, we need to carefully study the structure of the projection operator $\Pi_{\text{\tiny $K\bigcap L$}}$ and its associated generalized Jacobian, which plays a fundamental role in the design of the SsN method.

\subsubsection{The computation of the projection operator $\Pi_{\text{\tiny $K\bigcap L$}}$}
\label{subsec:projection operator}

In this part, we focus on how to compute the projection operator $\Pi_{\text{\tiny $K\bigcap L$}}(\vv)$ efficiently, where $\vv \in \mathbb{R}^{n}$ is a given vector. Firstly, it follows from the definition of $K$ and $L$ in \eqref{Def L and K} that $\Pi_{\text{\tiny $K\bigcap L$}}(\vv)$ is the solution of the following optimization problem
\begin{equation}\label{Def proj_KL001}
\begin{split}
\min_{\vx\in \mathbb{R}^n}\hspace*{2mm} &\frac{1}{2}||\vx-\vv||^{2}\\
\text{s.t.}\hspace*{2mm} &\va^{T}\vx=d,\\
& \vl \leq \vx \leq \vu.
\end{split}
\end{equation}
Furthermore, by introducing a slack variable $y$,  the problem \eqref{Def proj_KL001} can be reformulated as
\begin{equation}\label{Def proj_KL}
\begin{split}
\min_{\vx\in \mathbb{R}^n, \vy\in \mathbb{R}^n}\hspace*{2mm} &\frac{1}{2}||\vx-\vv||^{2}+\delta_{K}(\vy)\\
\text{s.t.}\hspace*{2mm} &\va^{T}\vx=d,\\
&\vx-\vy=0,\\
\end{split}
\end{equation}
where $\delta_{K}$ is an indicator function for the polyhedral convex set $K$.
Correspondingly, the dual problem of \eqref{Def proj_KL} is
\begin{equation}\label{Def Dual proj_KL}
\begin{split}
\min_{\lambda\in \mathbb{R}, \vz\in \mathbb{R}^n}\hspace*{2mm} &\frac{1}{2}||\vv-\lambda\va -\vz||^{2}+\delta^{*}_{K}(\vz)-\frac{1}{2}||\vv||^{2}+\lambda d,\\
\end{split}
\end{equation}
and the associated KKT condition is
\begin{equation}\label{Def KKT for proj_KL}
\begin{split}
\vx&=\Pi_{K}(\vv-\lambda\va),\ \ \vz=\vv-\lambda\va-\vx,\\
\va^{T}\vx&=d,\ \ \vx=\vy,\ \ (\vx,\vy,\lambda,\vz) \in \mathbb{R}^n \times K \times\mathbb{R} \times\mathbb{R}^n.
\end{split}
\end{equation}
Hence, combining \eqref{Def proj_KL}, \eqref{Def Dual proj_KL} and \eqref{Def KKT for proj_KL}, it is sufficient for us to obtain $\Pi_{\text{\tiny $K\bigcap L$}}(\vv)$ by solving the problem \eqref{Def Dual proj_KL}. Let

$$\varphi(\lambda):=\inf_{\vz}\big{\{}\frac{1}{2}||\vv-\lambda\va-\vz||^{2}+\delta^{*}_{K}(\vz)+\lambda d \big{\}},$$
which implies that the optimal solution $(\hat{\lambda},\hat{\vz})$ of the problem \eqref{Def Dual proj_KL} is given by
\begin{equation}\label{Def Obtain bar_lambda}
\begin{split}
\hat{\lambda}&=\arg\min\{\varphi(\lambda)|\lambda \in \mathbb{R} \},\\
\end{split}
\end{equation}
and
$$\hat{\vz}=\vv-\hat{\lambda}\va-\Pi_{K}(\vv-\va\hat{\lambda}),$$
respectively. Since the function $\varphi$ is convex and continuously differentiable with its gradient given by
$$\nabla\varphi(\lambda)=-\va^{T}\Pi_{K}(\vv-\lambda\va)+d,$$
the optimal solution $\hat{\lambda}$ of \eqref{Def Obtain bar_lambda} is the zero point of $\nabla\varphi$, i.e.,
\begin{equation}\label{Def nonsmooth piecewise affine y}
\begin{split}
\nabla\varphi(\hat{\lambda})=0.
\end{split}
\end{equation}
It follows by \cite{Helgason1980A} that $\nabla\varphi(\lambda)$ is a continuous non-decreasing piecewise linear function which has break points in the following set
\begin{equation*}\label{Def break points}
T=\left\{\frac{v_i-u_{i}}{a_i}, \frac{v_i-l_{i}}{a_i}\ |\ i=1,\ldots,n\right\}.
\end{equation*}
Moreover, the range of $\nabla\varphi(\lambda)$ is a closed and bounded interval
$$\bigg{[}d+\sum_{k\in I_{a_{-}}}{l_{k}|a_{k}|}-\sum_{j\in I_{a_{+}}}{u_{j}|a_{j}|},d+\sum_{k\in I_{a_{-}}}{u_{k}|a_{k}|}-\sum_{j\in I_{a_{+}}}{l_{j}|a_{j}|}\bigg{]},$$
where $I_{a_{+}}=\{i\,|\, a_i>0, i=1,\ldots,n\}$ and $I_{a_{-}}=\{i\,|\, a_i<0, i=1,\ldots,n\}$.

Next, we introduce the algorithm in \cite{Helgason1980A} to obtain the zero point of the equation \eqref{Def nonsmooth piecewise affine y}. The procedure consists of a binary search among the $2n$ break points until bracketing $\hat{\lambda}$ between two breakpoints. The algorithm can be summarized as follows

\begin{algorithm}[H]
\caption{: The breakpoint search algorithm for solving \eqref{Def nonsmooth piecewise affine y} \cite{Helgason1980A}}\label{alg:breakpoint}
\begin{algorithmic}
\State \textbf{Initialization: } Sort all break points given in \eqref{Def break points} with an ascending order
$$t^{(1)} \leq t^{(2)} \leq \ldots\leq t^{(2n)},\ \left(t^{(i)}\in T,\ \forall i=1,\cdots,2n\right)$$
Given $I_l=1$, $I_{u}=2n$, $f_l=\nabla\varphi(t^{(I_l)})$, and $f_u=\nabla\varphi(t^{(I_u)})$.\\
\textbf{While } $I_{u}-I_{l} >1$ and $t^{(I_u)}>t^{(I_l)}$\\
\ \ \ \ \ \ $I_m=[\frac{I_l+I_u}{2}]$, $f_m=\nabla\varphi(t^{(I_m)})$.\\
 \ \ \ \ \ \ \textbf{if} $f_m \geq 0$\\
 \ \ \ \ \ \ \ \ \ \  $I_{u}=I_m$, $f_u=f_m$.\\
 \ \ \ \ \ \ \textbf{else} \\
  \ \ \ \ \ \ \ \ \ \ $I_{l}=I_m$, $f_l=f_m$.\\
 \ \ \ \ \ \ \textbf{end}\\
 \textbf{end}\\
 \textbf{if} $f_u=f_l=0$\\
 \ \ \ \ $\hat{\lambda}=t^{(I_l)}$.\\
 \textbf{else} \\
 \ \ \ \ $\hat{\lambda}=t^{(I_l)}-\frac{f_l}{f_u-f_l}(t^{(I_u)}-t^{(I_l)})$.\\
\textbf{end}\\
\end{algorithmic}
\end{algorithm}

\begin{remark}
Compared with the bisection method used in \cite{ito2017unified}, the breakpoint search algorithm in our paper can avoid the complicated operations on some index sets. Hence, it is often faster than the bisection method in the practical implementation.
\end{remark}
Finally, suppose that the optimal solution $\hat{\lambda}$ of \eqref{Def Obtain bar_lambda} is obtained by Algorithm \ref{alg:breakpoint}. Then, in combination with the KKT condition \eqref{Def KKT for proj_KL}, we have
\begin{equation*}
\begin{split}
  \Pi_{\text{\tiny $K\bigcap L$ }}(\vv)&=\Pi_{\text{\tiny $K$ }}(\vv-\hat{\lambda}\va)\\
  &=\left(\Pi_{[l_{1},u_{1}]}(v_{1}-\hat{\lambda}a_{1}),
\ldots,\Pi_{[l_{n},u_{n}]}(v_{n}-\hat{\lambda}a_{n})\right)^{T},
\end{split}
\end{equation*}
where
\begin{equation*}
\Pi_{[l_{i},u_{i}]}(v_{i}-\hat{\lambda}a_{i})=
\begin{cases}
u_{i} & \  \text{if}\ v_{i}-\hat{\lambda}a_{i} \geq u_{i},  \\
v_{i}-\hat{\lambda}a_{i} & \  \text{if}\ l_{i}< v_{i}-\hat{\lambda}a_{i} < u_{i},  \\
l_{i} & \ \text{if}\ v_{i}-\hat{\lambda}a_{i} \leq l_{i},
\end{cases}
\ \ (i=1,\ldots,n).
\end{equation*}

\subsubsection{The computation of the HS-Jacobian of $\Pi_{\text{\tiny $K\bigcap L$ }}$}
\label{subsec:HS-Jacobian}

In the following, we proceed to find the HS-Jacobian of $\Pi_{\text{\tiny $K\bigcap L$}}$ at a given point $\vv \in \mathbb{R}^{n}$.
For the sake of clarity, we rewrite the box constraint in \eqref{Def proj_KL001} as a general linear inequality constraint so that the problem \eqref{Def proj_KL001} becomes

\begin{equation}\label{Def proj_KL HSj}
\begin{split}
\min_{\vx\in \mathbb{R}^n}\hspace*{2mm} &\frac{1}{2}||\vx-\vv||^{2}\\
\text{s.t.}\hspace*{2mm} &\va^{T}\vx=d,\\
&A\vx\geq \vg,
\end{split}
\end{equation}
where $A=[I_n^{T},-I_n^{T}]^{T} \in \mathbb{R}^{2n \times n}$ and $\vg=[\vl^{T},-\vu^{T}]^{T} \in \mathbb{R}^{2n}$. Denote
\begin{equation}\label{Def_active_set}
\mathcal{I}(\vv)=\{i|\ A_{i}\Pi_{\text{\tiny $K\bigcap L$ }}(\vv)=g_{i},i=1,\ldots,2n\},
\end{equation}
and $r=|\mathcal{I}(\vv)|$, the cardinality of $\mathcal{I}(\vv)$, where $A_{i}$ is the $i$th row of the matrix $A$.
Then, with the notation introduced above, we may compute the HS-Jacobian by the following results.
 \begin{theorem}\label{theorem3_1}
For any $\vv \in \mathbb{R}^{n}$, let $\mathcal{I}(\vv)$ be given in \eqref{Def_active_set} and
\begin{equation}\label{Def_Sigma}
  \Sigma:=I_n-\textrm{Diag}(\vtheta) \in \mathbb{R}^{n \times n},
\end{equation}
where $\vtheta \in \mathcal{R}^{n}$ is defined as
\begin{equation*}
\theta_{i}=
\begin{cases}
1 & \  i \in \mathcal{I}(\vv),  \\
0 & \ \text{otherwise}.
\end{cases}
\ \ i=1,\ldots,n.
\end{equation*}
Then
\begin{equation}\label{Def_HS_P}
P=
\begin{cases}
\Sigma(I_n-\frac{1}{\va^{T}\Sigma \va}\va \va^{T}) \Sigma, & \text{if}\  \va^{T}\Sigma \va \neq 0,  \\
\Sigma, & \text{otherwise},
\end{cases}
\end{equation}
is the HS-Jacobian of $\Pi_{\text{\tiny $K\bigcap L$ }}$ at $\vv$.
\end{theorem}

\begin{proof}
Firstly, it follows from Theorem $1$ in \cite{Li2017Birkhoff} that the following matrix
\begin{equation}\label{Def P0}
\begin{split}
P_{0}=I_n-\left[
         \begin{array}{cc}
            A^{T}_{\mathcal{I}(\vv)} & \va \\
         \end{array}
       \right]
       \bigg{(}
       \left[
         \begin{array}{c}
           A_{\mathcal{I}(\vv)} \\
           \va^{T} \\
         \end{array}
       \right]
       \left[
         \begin{array}{cc}
            A^{T}_{\mathcal{I}(\vv)} & \va \\
         \end{array}
       \right]
       \bigg{)}^{\dagger}
       \left[
         \begin{array}{c}
           A_{\mathcal{I}(\vv)} \\
           \va^{T} \\
         \end{array}
       \right]
\\
\end{split}
\end{equation}
is the HS-Jacobian of $\Pi_{\text{\tiny $K\bigcap L$}}$ at $\vv$, where $A_{\mathcal{I}(\vv)}$  is the matrix consisting of the rows of $A$ indexed by $\mathcal{I}(\vv)$. Moreover, the definitions of $\mathcal{I}(\vv)$ and $A_{\mathcal{I}(\vv)}$ imply that
$A_{\mathcal{I}(\vv)}A_{\mathcal{I}(\vv)}^{T}=I_r\ \text{and}\ A_{\mathcal{I}(\vv)}^{T}A_{\mathcal{I}(\vv)}=I_{n}-\Sigma.$

Secondly, we focus on the calculation of the Moore-Penrose pseudo-inverse in \eqref{Def P0}. Above all, it is easy to verify that
\begin{equation*}
\begin{split}
\text{det}\bigg{(}
       \left[
         \begin{array}{c}
           A_{\mathcal{I}(\vv)} \\
           \va^{T} \\
         \end{array}
       \right]
       \left[
         \begin{array}{cc}
            A^{T}_{\mathcal{I}(\vv)} & \va \\
         \end{array}
       \right]
       \bigg{)}=\va^{T}\Sigma \va.
\\
\end{split}
\end{equation*}
Then, on one hand, if $\va^{T}\Sigma \va \neq 0$, we have that
\begin{equation*}\label{Def P0 case2}
\begin{split}
P_{0}=&I_n-\left[
         \begin{array}{cc}
            A^{T}_{\mathcal{I}(\vv)} & \va \\
         \end{array}
       \right]
       \bigg{(}
       \left[
         \begin{array}{c}
           A_{\mathcal{I}(\vv)} \\
           \va^{T} \\
         \end{array}
       \right]
       \left[
         \begin{array}{cc}
            A^{T}_{\mathcal{I}(\vv)} & \va \\
         \end{array}
       \right]
       \bigg{)}^{-1}
       \left[
         \begin{array}{c}
           A_{\mathcal{I}(\vv)} \\
           \va^{T} \\
         \end{array}
       \right],\\
       =&I_n-\left[
         \begin{array}{cc}
            A^{T}_{\mathcal{I}(\vv)} & \va \\
         \end{array}
       \right]
       \left(
         \begin{array}{cc}
           I_{r}+\frac{1}{\va^{T}\Sigma\va}A_{\mathcal{I}(\vv)}\va\va^{T}A_{\mathcal{I}(\vv)}^{T} & -\frac{1}{\va^{T}\Sigma\va}A_{\mathcal{I}(\vv)}\va \\
           -\frac{1}{\va^{T}\Sigma\va}\va^{T}A_{\mathcal{I}(\vv)}^{T} & \frac{1}{\va^{T}\Sigma\va} \\
         \end{array}
       \right)
       \left[
         \begin{array}{c}
           A_{\mathcal{I}(\vv)} \\
           \va^{T} \\
         \end{array}
       \right],\\
       =&\Sigma(I_n-\frac{1}{\va^{T}\Sigma
       \va}\va \va^{T}) \Sigma,
\end{split}
\end{equation*}
where the last equality holds because $\Sigma=I_n-A_{\mathcal{I}(\vv)}^{T}A_{\mathcal{I}(\vv)}$ and $\Sigma=\Sigma^{2}$.

On the other hand, if $\va^{T}\Sigma \va = (\Sigma \va)^{T}\Sigma \va=0$, i.e., $\Sigma \va = \vzero$, then
\begin{equation}\label{Def P0 case3}
\begin{split}
P_{0}=&I_n-\left[
         \begin{array}{cc}
            A^{T}_{\mathcal{I}(\vv)} & \va \\
         \end{array}
       \right]
       \left(
         \begin{array}{cc}
           I_{r}-\frac{\va^{T}\va+2}{(1+\va^{T}\va)^2}A_{\mathcal{I}(\vv)}\va\va^{T}A_{\mathcal{I}(\vv)}^{T} & \frac{1}{(1+\va^{T}\va)^2}A_{\mathcal{I}(\vv)}\va \\
           \frac{1}{(1+\va^{T}\va)^2}\va^{T}A_{\mathcal{I}(\vv)}^{T} & \frac{\va^{T}\va}{(1+\va^{T}\va)^2} \\
         \end{array}
       \right)
       \left[
         \begin{array}{c}
           A_{\mathcal{I}(\vv)} \\
           \va^{T} \\
         \end{array}
       \right]\\
       =&\Sigma,
\end{split}
\end{equation}
where the first equality follows from the definition of the Moore-Penrose pseudo-inverse and \eqref{Def P0}, and the last equality is owing to $A_{\mathcal{I}(\vv)}^{T}A_{\mathcal{I}(\vv)}\va=\va$. Thus, the
proof is completed.
\end{proof}

\subsubsection{The SsN method for the inner subproblem \eqref{eq inner problem1}}

In this part, we formally present the SsN method for the subproblem \eqref{eq inner problem1}. Recall that we need to solve the convex subproblem in each iteration of Algorithm \ref{Def Algorithm ALM}, i.e.,
$$\min_{\vw \in \text{Ran}(Q)}\left\{\psi^{k}(\vw):=\frac{1}{2}\langle \vw,Q\vw\rangle+\frac{1}{2\sigma_{k}}\left(||\vu(\vw)||^2-||\vu(\vw)-\Pi_{\text{\tiny $K\bigcap L$ }}(\vu(\vw))||^2\right)-\frac{1}{2\sigma_{k}}||\vx^{k}||^{2}\right\}.$$
Furthermore, note that
\begin{equation}\label{Def_differential_of_psi(w)}
\nabla\psi^{k}(\vw)=Q\vw-Q\Pi_{\text{\tiny $K\bigcap L$ }}(\vu(\vw)),
\end{equation}
which implies that the optimal solution $\bar{\vw}$ can be obtained through solving the following nonsmooth piecewise affine equation
\begin{equation}\label{Def_nonsmooth_equation}
\nabla\psi^{k}(\vw)=0,\ \ \ \vw \in \text{Ran}(Q).
\end{equation}

Let $\vw \in \text{Ran}(Q)$ be any given point. We define the following operator
$$\hat{\partial}^{2}\psi^{k}(\vw):=Q+\sigma_{k} Q\mathscr{P}(\vu(\vw))Q,$$
where the multivalued mapping $\mathscr{P}:\mathbb{R}^{n}\rightrightarrows \mathbb{R}^{n \times n}$ is the HS-Jacobian of $\Pi_{\text{\tiny $K\bigcap L$}}$ \cite{Han1997Newton}.

Now we are ready to state the SsN method as below.
\begin{algorithm}[H]
\caption{: the SsN method for the problem \eqref{eq inner problem1}} \label{alg:SsN}
\begin{algorithmic}
\State Given $\mu \in (0,\frac{1}{2}), \delta \in (0,1), \tau \in (0,1]$ and $\eta\in (0,1)$. Given an initial point $\vw^0 \in \mathbb{R}^{n}$. For $j=0,1,2,\ldots,$ iterate the following steps:
\State \textbf{Step 1. } Let $\mathcal{M}_{j}:=Q+\sigma Q\mathcal{P}_{j}Q$, where $\mathcal{P}_{j} \in \mathscr{P}(\vu(\vw^{j}))$. Apply the conjugate gradient (CG) method to find an approximate solution $d_{w^{j}}$ to the following linear system
\begin{equation}\label{Newton_Linear_System}
\mathcal{M}_{j}\vd_{\vw^{j}}+\nabla\psi^{k}(\vw^{j})=0, \ \ \ \ \vd_{\vw^{j}} \in \text{Ran}(Q)
\end{equation}
such that
$$||\mathcal{M}_{j}\vd_{\vw^{j}}+\nabla\psi^{k}(\vw^{j})||\leq \text{min}(\eta,||\nabla\psi^{k}(\vw^{j})||^{1+\tau}).$$
\State \textbf{Step 2. } (Line search) Set $\alpha_{j}=\delta^{m_j}$, where $m_j$ is the first nonneigative integer $m$ for which
$$\psi^{k}(\vw^{j}+\delta^{m}\vd_{\vw^{j}})\leq\psi^{k}(\vw^{j})+\mu\delta^{m}\langle \nabla\psi^{k}(\vw^{j}), \vd_{\vw^{j}} \rangle.$$
\State \textbf{Step 3. }Set $\vw^{j+1}=\vw^{j}+\alpha_{j}\vd_{\vw^{j}}$.
\end{algorithmic}
\end{algorithm}

\begin{remark}
  \ \\
 (i) Since the matrix $\mathcal{M}_{j}$ is positive definite in $\text{Ran}(Q)$ and $\nabla\psi^{k}(\vw^{j})\in \text{Ran}(Q)$, the linear system \eqref{Newton_Linear_System} has a unique solution in $\text{Ran}(Q)$.\\
 (ii) For the SVM problems with nonlinear kernels, we usually adopt some easy-to-implement algorithm to warm start the SSsNAL method. This process is very essential because if we choose an arbitrary initial point, when we implement Algorithm \ref{alg:SsN} we can hardly get the sparse structure in the early Newton iterations and the cost of the Newton step may be very high. We will present some necessary details in Section \ref{section_4}.
\end{remark}

Before analyzing the local convergence rate of the SsN method, we first consider the strong semismoothness of $\nabla\psi^{k}$. For the definitions of semismoothness and $\gamma$-order semismoothness, one may see the below.

\begin{definition}
(semismoothness) \cite{Kummer1988Newton,Qi1993A} Let $\mathcal{O} \subseteq \mathbb{R}^{n}$ be an open set, $\mathcal{K}:\mathcal{O}\subseteq \mathbb{R}^{n}\rightrightarrows \mathbb{R}^{n \times m}$ be a nonempty and compact valued, upper-semicontinuous set-valued mapping, and $F:\mathcal{O}\rightarrow \mathbb{R}^{n}$ be a locally Lipschitz continuous function. $F$ is said to be semismooth at $x \in O$ with respect to the multifunction $\mathcal{K}$ if $F$ is directionally differentiable at $x$ and for any $V \in \mathcal{K}(x+\Delta x)$ with $\Delta x\rightarrow 0$,
\begin{equation*}
  F(x+\Delta x)-F(x)-V\Delta x=o(||\Delta x||).
\end{equation*}
Let $\gamma$ be a positive constant. $F$ is said to be $\gamma$-order (strongly, if $\gamma=1$) semismooth at $X$ with respect to $\mathcal{K}$ if $F$ is directionally differentiable at $x$ and for any $V \in \mathcal{K}(x+\Delta x)$ with $\Delta x\rightarrow 0$,
\begin{equation*}
  F(x+\Delta x)-F(x)-V\Delta x=O(||\Delta x||^{1+\gamma}).
\end{equation*}
\end{definition}

According to \cite[Lemma 2.1]{Han1997Newton} and \cite[ Theorem 7.5.17]{Facchinei2003Finite}, we can immediately obtain that $\nabla\psi^{k}$ is strongly semismooth at $\vw$ with respect to $\hat{\partial}^{2}\psi^{k}$.
\begin{theorem}\label{theorem_SsN_convergence}
Let $\{\vw^j\}$ be the infinite sequence generated by Algorithm \ref{alg:SsN}. Then $\{\vw^j\}$ converges to the unique optimal solution $\tilde{\vw} \in \text{Ran}(Q)$ to the problem \eqref{Def_nonsmooth_equation} and
$$||\vw^{j+1}-\tilde{\vw}||=O(||\vw^{j}-\tilde{\vw}||^{1+\tau}).$$
\end{theorem}

\subsubsection{An efficient implementation for solving the linear system \eqref{Newton_Linear_System}}
\label{subsubsec:linear system}
 In this part, we discuss how to solve the linear system of equations \eqref{Newton_Linear_System} efficiently. In Algorithm \ref{alg:SsN}, the most time-consuming step is solving the linear system \eqref{Newton_Linear_System}, so it is essential to make full use of the the sparse structure of the coefficient matrix to design an efficient algorithm to obtain  the descent direction $\vd_{\vw^{j}}$. For the sake of convenience, we ignore the subscript and rewrite the linear system \eqref{Newton_Linear_System} as follows
\begin{equation}\label{Newton system 1}
 \left(Q+\sigma Q\mathcal{P}Q\right)\vd=-\nabla\psi^{k}(\vw), \ \ \ \ \vd \in \text{Ran}(Q),
\end{equation}
where $\mathcal{P}\in \mathscr{P}(\vu(\vw))$ and $\vu(\vw)=\vx^{k}-\sigma_{k}(Q\vw+\vc)$. Although $Q$ may be singular, the subspace constraint and the fact that $\nabla\psi^{k}(\vw)\in \text{Ran}(Q)$ together imply that there exists a unique solution to the linear system \eqref{Newton system 1}. It is extraordinarily time-consuming to solve \eqref{Newton system 1} directly when the dimension of $Q$ is large. However, we only need to update $Q\vd$ and $\vd^{T}Q\vd$ in every iteration of Algorithm \ref{alg:SsN}. Hence, as shown in the next proposition, instead of computing the solution $\hat{\vd}$ of \eqref{Newton system 1} directly, we may choose to solve a relatively much smaller linear system to obtain $Q\hat{\vd}$ and $\hat{\vd}^{T}Q\hat{\vd}$ by deliberately employing the sparse structure of the HS-Jacobian.
\begin{proposition}\label{proposition5}
Let the index set $\mathcal{J}=\{1,\cdots,n\}/\mathcal{I}(\vu(\vw))$, where $\mathcal{I}(\vu(\vw))$ is defined by \eqref{Def_active_set} and $p=|\mathcal{J}|$ denotes the cardinality of $\mathcal{J}$. Moreover, Let $\Sigma \in \mathbb{R}^{n \times n}$ be a diagonal matrix whose $i$-th diagonal element $\Sigma_{ii}$ is equal to $1$ if $i \in \mathcal{J}$, otherwise $\Sigma_{ii}$ is equal to $0$. Assume that $\hat{\vd}\in \text{Ran}(Q)$ is the solution to the linear system \eqref{Newton system 1}.

$(a)$ If $\va^{T}\Sigma \va\neq 0$, then we have
\begin{equation}\label{Def_Qd_start1}
\begin{split}
  Q\hat{\vd}=&Q^{T}_{\mathcal{J}}\vv_{\mathcal{J}}(\vw)-\nabla\psi^{k}(\vw)+\sigma\frac{ \va^{T}_{\mathcal{J}}Q_{\mathcal{J}\mathcal{J}}\vv_{\mathcal{J}}(\vw)-
  \va^{T}_{\mathcal{J}}\nabla\psi^{k}_{\mathcal{J}}(\vw)}{\va_{\mathcal{J}}^{T} \va_{\mathcal{J}}-\sigma\va_{\mathcal{J}}^{T} Q_{\mathcal{J}\mathcal{J}}\va_{\mathcal{J}}}Q^{T}_{\mathcal{J}} \va_{\mathcal{J}}
  \end{split}
\end{equation}
and
\begin{equation}\label{Def_dQd}
\begin{split}
  \hat{\vd}^{T}Q\hat{\vd}=&\vv^{T}_{\mathcal{J}}(\vw)Q_{\mathcal{J}\mathcal{J}}\vv_{\mathcal{J}}(\vw)
  -2\vv^{T}_{\mathcal{J}}(\vw)\nabla\psi^{k}_{\mathcal{J}}(\vw)+(\vs(\vw))^{T}Q\vs(\vw)\\
  &+\frac{2\sigma\va_{\mathcal{J}}^{T} \va_{\mathcal{J}}-\sigma^{2}\va_{\mathcal{J}}^{T} Q_{\mathcal{J}\mathcal{J}}\va_{\mathcal{J}}}{\left(\va_{\mathcal{J}}^{T} \va_{\mathcal{J}}-\sigma\va_{\mathcal{J}}^{T} Q_{\mathcal{J}\mathcal{J}}\va_{\mathcal{J}}\right)^{2}}
  (\va^{T}_{\mathcal{J}}Q_{\mathcal{J}\mathcal{J}}\vv_{\mathcal{J}}(\vw)-
  \va^{T}_{\mathcal{J}}\nabla\psi^{k}_{\mathcal{J}}(\vw))^{2},
 \end{split}
\end{equation}
where $\vs(\vw)=\vw-\Pi_{\text{\tiny $K\bigcap L$ }}(\vu(\vw))$ and $\vv_{\mathcal{J}}(\vw) \in \mathbb{R}^{p}$ is the solution to the following linear system
\begin{equation}\label{Def_small_linear_eq}
\begin{split}
  \left(\frac{1}{\sigma}I_{p}+Q_{\mathcal{J}\mathcal{J}}+\frac{\sigma Q_{\mathcal{J}\mathcal{J}} \va_{\mathcal{J}} \va_{\mathcal{J}}^{T}Q_{\mathcal{J}\mathcal{J}} }{\va_{\mathcal{J}}^{T} \va_{\mathcal{J}}-\sigma\va_{\mathcal{J}}^{T} Q_{\mathcal{J}\mathcal{J}}\va_{\mathcal{J}}}\right)\vv_{\mathcal{J}}(\vw)=\left(I_{p}+\frac{\sigma Q_{\mathcal{J}\mathcal{J}} \va_{\mathcal{J}} \va_{\mathcal{J}}^{T} }{\va_{\mathcal{J}}^{T} \va_{\mathcal{J}}-\sigma\va_{\mathcal{J}}^{T} Q_{\mathcal{J}\mathcal{J}}\va_{\mathcal{J}}}\right)\nabla\psi^{k}_{\mathcal{J}}(\vw),\\
  \end{split}
\end{equation}
and $Q_{\mathcal{J}\mathcal{J}}\in \mathbb{R}^{p \times p}$ is the submatrix of $Q$ with those rows and columns in $\mathcal{J}$. Moreover, $\nabla\psi^{k}_{\mathcal{J}}(\vw) \in \mathbb{R}^{p}$, $\va_{\mathcal{J}} \in \mathbb{R}^{p}$ and $Q_{\mathcal{J}}\in \mathbb{R}^{p \times n}$ are matrices consisting of the rows of $\nabla\psi^{k}(\vw)$, $\va$ and $Q$ indexed by $\mathcal{J}$, respectively.

$(b)$ If $\va^{T}\Sigma \va = 0$, then we have
\begin{equation}\label{Def_Qd_start1_s}
\begin{split}
  Q\hat{\vd}=&Q^{T}_{\mathcal{J}}\vv_{\mathcal{J}}(\vw)-\nabla\psi^{k}(\vw)
  \end{split}
\end{equation}
and
\begin{equation}\label{Def_dQd_s}
\begin{split}
  \hat{\vd}^{T}Q\hat{\vd}=&\vv^{T}_{\mathcal{J}}(\vw)Q_{\mathcal{J}\mathcal{J}}\vv_{\mathcal{J}}(\vw)
  -2\vv^{T}_{\mathcal{J}}(\vw)\nabla\psi^{k}_{\mathcal{J}}(\vw)+(\vs(\vw))^{T}Q\vs(\vw),\\
 \end{split}
\end{equation}
where $Q_{\mathcal{J}\mathcal{J}}$, $Q_{\mathcal{J}}$, $\nabla\psi^{k}_{\mathcal{J}}(\vw)$ and $\va_{\mathcal{J}}$ are same as those in $(a)$ and $\vv_{\mathcal{J}}(\vw) \in \mathbb{R}^{p}$ is the solution to the following linear system
\begin{equation}\label{Def_small_linear_eq_s}
\begin{split}
  \left(\frac{1}{\sigma}I_{p}+Q_{\mathcal{J}\mathcal{J}}\right)\vv_{\mathcal{J}}(\vw)=\nabla\psi^{k}_{\mathcal{J}}(\vw).\\
  \end{split}
\end{equation}

\end{proposition}

\begin{proof}
To prove part $(a)$, combining with Theorem \ref{theorem3_1} and the condition $\va^{T}\Sigma \va\neq 0$, we know that $\mathcal{P}=\Sigma(I_n-\frac{1}{\va^{T}\Sigma \va}\va \va^{T}) \Sigma$ is an element in the HS-Jacobian of $\Pi_{\text{\tiny $K\bigcap L$ }}$ at $\vu(\vw)$. From \eqref{Def_differential_of_psi(w)}, it is not difficult to establish
Then according to \eqref{Def_differential_of_psi(w)}, we have $\nabla\psi^{k}(\vw)=Q\vs(\vw)$. Thus, without loss of generality, we assume that $Q$ can be decomposed as $Q=LL^{T}$ where $L \in \mathbb{R}^{n \times r}$ is a full collum rank matrix and $r=\text{Rank}(Q)$. Then, substituting $Q=LL^{T}$ into \eqref{Newton system 1}, we obtain that
\begin{equation}\label{Newton system 2}
 L\left(I_{r}+\sigma L^{T}\mathcal{P}L\right)L^{T}\vd=-LL^{T}\vs(\vw),\ \ \ \ \vd \in \text{Ran}(Q).
\end{equation}
Since $L$ has a full collum rank and $\mathcal{P}=\Sigma(I_n-\frac{1}{\va^{T}\Sigma \va}\va \va^{T}) \Sigma$, \eqref{Newton system 2} is further equivalent to
\begin{equation}\label{Newton system 3}
 \left(I_{r}+\sigma L^{T}\Sigma(I_n-\frac{1}{\va^{T}\Sigma \va}\va \va^{T}) \Sigma L\right)L^{T}\vd=-L^{T}\vs(\vw),\ \ \ \ \vd \in \text{Ran}(Q).
\end{equation}
Since $\hat{\vd}\in \text{Ran}(Q)$ is the solution to the linear system \eqref{Newton system 1}, we have
\begin{equation}\label{Def_Ld}
\begin{split}
 L^{T}\hat{\vd}=&-\left(I_{r}+\sigma L^{T}\Sigma(I_n-\frac{1}{\va^{T}\Sigma \va}\va \va^{T}) \Sigma L\right)^{-1}L^{T}\vs(\vw)\\
  =&-\left(I_{r}-\frac{\sigma }{\va_{\mathcal{J}}^{T} \va_{\mathcal{J}}}L_{\mathcal{J}}^{T} \va_{\mathcal{J}} \va_{\mathcal{J}}^{T}L_{\mathcal{J}}+\sigma L_{\mathcal{J}}^{T}L_{\mathcal{J}}\right)^{-1}L^{T}\vs(\vw)\\
  =&-\left(I_{r}+\frac{\sigma L_{\mathcal{J}}^{T} \va_{\mathcal{J}} \va_{\mathcal{J}}^{T}L_{\mathcal{J}} }{\va_{\mathcal{J}}^{T} \va_{\mathcal{J}}-\sigma\va_{\mathcal{J}}^{T} L_{\mathcal{J}}L_{\mathcal{J}}^{T}\va_{\mathcal{J}}}\right)L^{T}\vs(\vw)+\left(I_{r}+\frac{\sigma L_{\mathcal{J}}^{T} \va_{\mathcal{J}} \va_{\mathcal{J}}^{T}L_{\mathcal{J}} }{\va_{\mathcal{J}}^{T} \va_{\mathcal{J}}-\sigma\va_{\mathcal{J}}^{T} L_{\mathcal{J}}L_{\mathcal{J}}^{T}\va_{\mathcal{J}}}\right)L_{\mathcal{J}}^{T}\\
  &\left(\frac{1}{\sigma}I_{p}+L_{\mathcal{J}}L^{T}_{\mathcal{J}}+\frac{\sigma L_{\mathcal{J}}L_{\mathcal{J}}^{T} \va_{\mathcal{J}} \va_{\mathcal{J}}^{T}L_{\mathcal{J}}L^{T}_{\mathcal{J}} }{\va_{\mathcal{J}}^{T} \va_{\mathcal{J}}-\sigma\va_{\mathcal{J}}^{T} L_{\mathcal{J}}L_{\mathcal{J}}^{T}\va_{\mathcal{J}}}\right)^{-1}L_{\mathcal{J}}\left(I_{r}+\frac{\sigma L_{\mathcal{J}}^{T} \va_{\mathcal{J}} \va_{\mathcal{J}}^{T}L_{\mathcal{J}} }{\va_{\mathcal{J}}^{T} \va_{\mathcal{J}}-\sigma\va_{\mathcal{J}}^{T} L_{\mathcal{J}}L_{\mathcal{J}}^{T}\va_{\mathcal{J}}}\right)L^{T}\vs(\vw)\\
   =&-L^{T}\vs(\vw)-\frac{\sigma L_{\mathcal{J}}^{T} \va_{\mathcal{J}} \va_{\mathcal{J}}^{T}\nabla\psi^{k}_{\mathcal{J}}(\vw)}{\va_{\mathcal{J}}^{T} \va_{\mathcal{J}}-\sigma\va_{\mathcal{J}}^{T} Q_{\mathcal{J}\mathcal{J}}\va_{\mathcal{J}}}+L_{\mathcal{J}}^{T}\left(I_{p}+\frac{\sigma \va_{\mathcal{J}} \va_{\mathcal{J}}^{T}Q_{\mathcal{J}\mathcal{J}} }{\va_{\mathcal{J}}^{T} \va_{\mathcal{J}}-\sigma\va_{\mathcal{J}}^{T} Q_{\mathcal{J}\mathcal{J}}\va_{\mathcal{J}}}\right)\\
  &\left(\frac{1}{\sigma}I_{p}+Q_{\mathcal{J}\mathcal{J}}+\frac{\sigma Q_{\mathcal{J}\mathcal{J}} \va_{\mathcal{J}} \va_{\mathcal{J}}^{T}Q_{\mathcal{J}\mathcal{J}} }{\va_{\mathcal{J}}^{T} \va_{\mathcal{J}}-\sigma\va_{\mathcal{J}}^{T} Q_{\mathcal{J}\mathcal{J}}\va_{\mathcal{J}}}\right)^{-1}\left(I_{p}+\frac{\sigma Q_{\mathcal{J}\mathcal{J}} \va_{\mathcal{J}} \va_{\mathcal{J}}^{T} }{\va_{\mathcal{J}}^{T} \va_{\mathcal{J}}-\sigma\va_{\mathcal{J}}^{T}Q_{\mathcal{J}\mathcal{J}}\va_{\mathcal{J}}}\right)\nabla\psi^{k}_{\mathcal{J}}(\vw)\\
   =&-L^{T}\vs(\vw)-\frac{\sigma L_{\mathcal{J}}^{T} \va_{\mathcal{J}} \va_{\mathcal{J}}^{T}\nabla\psi^{k}_{\mathcal{J}}(\vw)}{\va_{\mathcal{J}}^{T} \va_{\mathcal{J}}-\sigma\va_{\mathcal{J}}^{T} Q_{\mathcal{J}\mathcal{J}}\va_{\mathcal{J}}}+L_{\mathcal{J}}^{T}\left(I_{p}+\frac{\sigma \va_{\mathcal{J}} \va_{\mathcal{J}}^{T}Q_{\mathcal{J}\mathcal{J}} }{\va_{\mathcal{J}}^{T} \va_{\mathcal{J}}-\sigma\va_{\mathcal{J}}^{T} Q_{\mathcal{J}\mathcal{J}}\va_{\mathcal{J}}}\right)\vv_{\mathcal{J}}(\vw)\\
 \end{split}
\end{equation}
where the first equality is thanks to \eqref{Newton system 2} and the nonsingularity of $\left(I_{r}+\sigma L^{T}\Sigma(I_n-\frac{1}{\va^{T}\Sigma \va}\va \va^{T}) \Sigma L\right)$; the second equality follows from the special $0$-$1$ structure of the diagonal matrix $\Sigma$ and the fact that $L^{T}\Sigma \Sigma L=L_{\mathcal{J}}^{T}L_{\mathcal{J}}$, $\va^{T}\Sigma \va=\va_{\mathcal{J}}^{T}\va_{\mathcal{J}}$ and $\va^{T}\Sigma L=\va_{\mathcal{J}}^{T}L_{\mathcal{J}}$, where $L_{\mathcal{J}}\in \mathbb{R}^{p \times r}$ is a matrix consisting of the rows of $L$ indexed by $\mathcal{J}$; the third equality is obtained by using the Sherman-Morrison-Woodbury formula \cite{Golub1996Matrix} twice; the fourth equality is due to $Q_{\mathcal{J}\mathcal{J}}=L_{\mathcal{J}}L_{\mathcal{J}}^{T}$ and $\nabla\psi^{k}_{\mathcal{J}}(\vw)=L_{\mathcal{J}}L^{T}\vs(\vw)$. Finally, it is immediately follows from $Q\hat{\vd}=LL^{T}\hat{\vd}$, $\hat{\vd}^{T}Q\hat{\vd}=(L^{T}\hat{\vd})^{T}L^{T}\hat{\vd}$ and \eqref{Def_Ld} that \eqref{Def_Qd_start1} and \eqref{Def_dQd} hold true, thus concluding the proof of part $(a)$.

As for part $(b)$, since $\va^{T}\Sigma\va=0$, it follows from Theorem \ref{theorem3_1} that $\mathcal{P}=\Sigma \in \mathscr{P}(\vu(\vw))$.
Similar to the proof for part $(a)$, we can obtain the desired results \eqref{Def_Qd_start1_s}-\eqref{Def_small_linear_eq_s}.
\end{proof}

\begin{remark}
In the proof of Proposition \ref{proposition5}, we always assume that $\va_{\mathcal{J}}^{T} \va_{\mathcal{J}}-\sigma\va_{\mathcal{J}}^{T} Q_{\mathcal{J}\mathcal{J}}\va_{\mathcal{J}}\neq 0$, i.e., the matrix $I_{r}-\frac{\sigma }{\va_{\mathcal{J}}^{T} \va_{\mathcal{J}}}L_{\mathcal{J}}^{T} \va_{\mathcal{J}} \va_{\mathcal{J}}^{T}L_{\mathcal{J}}$ in \eqref{Def_Ld} is invertible. Actually this assumption is reasonable because $\va_{\mathcal{J}}^{T} \va_{\mathcal{J}}=\va^{T}\Sigma\va\neq 0$ and $\sigma$ can be adjusted in an appropriate way in the implementation of the algorithm to avoid $\va_{\mathcal{J}}^{T} \va_{\mathcal{J}}-\sigma\va_{\mathcal{J}}^{T} Q_{\mathcal{J}\mathcal{J}}\va_{\mathcal{J}}= 0$.
\end{remark}

From the above discussion, we can see that the computational costs for solving the Newton linear system \eqref{Newton_Linear_System} are reduced significantly from $O(n^3)$ to $O(p^3)$. According to the updating rule for the primal variable $\vx^{k}$ in \eqref{Def_new_update_xkplus},
the number $p$ is also equal to the number of the unbounded support vectors \cite{abe2005support} in the SVMs, which is usually much smaller than the number of samples $n$ \cite{Keerthi2006Building}. So we can always solve the linear system \eqref{Def_small_linear_eq} or \eqref{Def_small_linear_eq_s} at very low costs.

\section{Numerical experiments}\label{section_4}

In this section, we demonstrate the performances of the SSsNAL method by solving the SVM problem \eqref{Problem_P1_classical1} on the benchmark datasets from the LIBSVM data \cite{Chang2011LIBSVM}. For a comparison, three state-of-the-art solvers for solving the SVM problems: the P2GP method \cite{Serafino2018A}, the LIBSVM \cite{Chang2011LIBSVM} and the FAPG method \cite{ito2017unified} are used to solve the same problems. Note that the P2GP method is a two-phase gradient-based method and its {\sc Matlab} code can be downloaded from \textit{https://github.com/diserafi/P2GP}. The LIBSVM implements the sequential minimal optimization method \cite{Platt1999Fast} which is available on \textit{https://www.csie.ntu.edu.tw/\textasciitilde cjlin/libsvm/index.html}, and the FAPG method is a general optimization algorithm based on an accelerated proximal gradient method and the authors shared their code to us. Moreover, we scale the LIBSVM datasets so that each sample $\tilde{\vx}_{i} \in [0,1]^{q}, (i=1,2,\ldots,n)$. The details of the LIBSVM datasets (the size of the problem, the proportion of nonzeros in the data and the type of the task) are presented in Table \ref{table1}.

All experiments are implemented in {\sc Matlab} R2018b on a PC with the Intel Core i7-6700HQ processors (2.60GHz) and 8 GB of physical memory.

\subsection{The stopping criterion}\label{sec_SC}

\label{subsec:parameters setting}
We use the following relative KKT residual as a stopping criterion for all the algorithms which only involves the variable $\vx$ in the primal problem \eqref{Problem_P1_classical1}
 \begin{equation}\label{Def_Rkkt}
\text{R}_{\text{KKT}}(\vx)=\frac{||\vx-\Pi_{\text{\tiny $K\bigcap L$}}(\vx-Q\vx-\vc)||}{1+||\vx||} < \textrm{Tol},
\end{equation}
where $\textrm{Tol}=10^{-3}\footnote{Note that our SSsNAL method is easier and faster to generate a high accurate solution, say, a solution with the KKT residual less than $10^{-6}$. But we found that the accuracy of the classification with the highly accurate solution can only improve very little, so here we only set Tol $= 10^{-3}$.}$. We also set the maximum numbers of iterations for the SSsNAL, FAPG, P2GP and LINBSVM to be $200$, $20000$, $20000$ and $20000$, respectively. Moreover, the initial values of variables $\vx^0$ and $\vw^0$ are set to be zero vectors in our SSsNAL method.

\subsection{Numerical experiments for the support vector classification problems with the linear kernel} \label{subsection_linear_kernel_case}

In this subsection, the performances of all the methods on the C-SVC problem \eqref{Def_Dual_of_C_SVM_Problem} with a linear kernel function $k(\vx,\vy)=\vx^{T}\vy$ are presented. For each LIBSVM data set, we adopt the ten-fold cross validation scheme to select the penalty parameter $C$. Moreover, the linear C-SVC model is trained on the training set and the testing data is used to verify the classification accuracy. As for those datasets which do not have the corresponding testing data, we randomly and uniformly sample $80\%$ of the data from the dataset as a training set, and use the remained $20\%$ as a testing set. In order to eliminate the effects of randomness, we repeat the above process $10$ times and obtain an average result.

\newpage
\begin{table}[h]
\centering
 \begin{tabular}{|c||rrrrrrccc|}
  \hline
        Data   & $n_{\text{train}}$ & ($n_{\text{train}}^{+},$  & $n_{\text{train}}^{-}$)& $n_{\text{test}}$ & ($n_{\text{test}}^{+},$  & $n_{\text{test}}^{-}$)& $q$ & density & type\\
 \hline
  splice
 &1000 & (517, & 483) & 2175  &(1131, & 1044)& 60 &1.000 & classification \\
  madelon
 &2000 & (1000, & 1000) & 600  &(300, & 300)& 500 &0.999& classification\\
  ijcnn1
  &35000 & (3415, & 31585) & 91701  &(8712, & 82989)& 22 &0.591& classification\\
  svmguide1
 &3089 & (1089, & 2000) & 4000  & (2000, & 2000)& 4 &0.997& classification\\
  svmguide3
 &1243 & (947, & 296) & 41  & (41, & 0)& 22 &0.805& classification\\
w1a
 &2477 & (72, & 2405) & 47272  & (1407, & 45865)& 300 &0.039& classification\\
 w2a
 &3470 & (107, & 3363) & 46279  & (1372, & 44907)& 300 &0.039& classification\\
 w3a
 &4912 & (143, & 4769) & 44837  & (1336, & 43501)& 300 &0.039& classification\\
 w4a
 &7366 & (216, & 7150) & 42383  & (1263, & 41120)& 300 &0.039& classification\\
 w5a
 &9888 & (281, & 9607) & 39861  & (1198, & 38663)& 300 &0.039& classification\\
 w6a
 &17188 & (525, & 16663) & 32551  & (954, & 31607)& 300 &0.039& classification\\
 w7a
 &5103 & (740, & 4363) & 25057  & (739, & 24318)& 300 &0.039& classification\\
 australian
 &690 & (307, & 383) & *  & (*, & *)& 14 & 0.874& classification\\
 breast-cancer
 &683 & (239, & 444) & *  & (*, & *)& 10 & 1.000& classification\\
 mushrooms
 &8124 & (3916, & 4208) & * & (*, & *)& 112 & 0.188& classification\\
 phishing
 &11055 & (6157, & 4898) & * & (*, & *)& 68 & 0.441& classification\\
 diabetes
 &768 & (500, & 268) & *  & (*, & *)& 8 & 0.999& classification\\
 ionosphere
 &651 & (225, & 126) & *  & (*, & *)& 34 & 0.884& classification\\
 heart
 &270 & (120, & 150) & *  & (*, & *)& 13 & 0.962& classification\\
 fourclass
 &862 & (307, & 555) & *  & (*, & *)& 2 & 0.983& classification\\
 colon-cancer
 &62 &  (22, & 40) & *  & (*, & *)& 2000 & 1.000& classification\\
 diabetes
 &768 & (500, & 268) & *  & (*, & *)& 8 & 0.999& classification\\
 a1a
 &1605 & (395, & 1210) & *  & (*, & *)& 119 & 0.116& classification\\
 a2a
 &2265 & (572, & 1693) & *  & (*, & *)& 119 & 0.115& classification\\
 a3a
 &3185 & (773, & 2412) & *  & (*, & *)& 122 & 0.114& classification\\
 a4a
 &4781 & (1188, & 3593) & *  & (*, & *)& 122 & 0.114& classification\\
 a6a
 &11220 & (2692, & 8528) & * & (*, & *)& 122 & 0.114& classification\\
  a7a
 &16100 & (3918, & 12182) & *  & (*, & *)& 122 & 0.115& classification\\
  cadata
 &11080 & (*, & *) & *  & (*, & *)& 8 & 1.000& regression\\
 abalone
 &4177 & (*, & *) & * & (*, & *)& 8 & 0.960& regression\\
 space\_ga
 &3107 & (*, & *) & *  & (*, & *)& 6 & 0.750& regression\\
  \hline
\end{tabular}
\caption{The details of the LIBSVM Datasets. `$*$' means that it does not have the corresponding data or the features of the training data and testing data are not equal to each other.}\label{table1}
\end{table}

\newpage

We report the numerical results in Table \ref{table2}. For the numerical results, we report the data set name (Data), the number of samples ($n$) and features ($q$), the value of penalty parameter (C), the number of unbounded support vectors (suppvec), the relative KKT residual ($\text{R}_{\text{KKT}}$), the computing time (Time), the iteration number (Iter) and the percentage of the classification accuracy on the testing data set (Accuracy).  Particularly, we use ``$s\ \text{sign}(t)|t|$" to denote a number of the form ``$s\times 10^{t}$" in all the following tables, e.g., $1.0$-$4$ denotes $1.0\times 10^{-4}$. The computing time here is in the format of ``hours:minutes:seconds''. Notably, ``00'' and ``T'' in the time column denote that the elapsed time is less than 0.5 second and more than 2 hours, respectively. For the SSsNAL method, we report the number of iterations in the form of ``$\text{Iter}_{\text{alm}}$($\text{Iter}_{\text{ssn}}$)'', where $\text{Iter}_{\text{alm}}$ and $\text{Iter}_{\text{ssn}}$ represent the number of outer iterations and the average number of inner iterations, respectively. Note that the P2GP method is an alternating method between two different phases, so we report its number of iterations in the form of
``$\text{Iter}_{\text{0}}$($\text{Iter}_{I}$,
$\text{Iter}_{I\!I}$)'', where $\text{Iter}_{\text{0}}$  represents the alternating number of phases. $\text{Iter}_{I}$ and $\text{Iter}_{I\!I}$ represent the the average number of iterations in phase I and II, respectively. Moreover, the reported unbounded support vectors are obtained from the SSsNAL method.

\begin{table}[htpb]
\begin{adjustbox}{addcode={\begin{minipage}{\width}}{\caption{Comparisons of the SSsNAL, FAPG, P2GP, and LINBSVM solvers for the C-SVC problem \eqref{Def_Dual_of_C_SVM_Problem} with the linear kernel function. In the table, ``a'' denotes the SSsNAL method, ``b'' denotes the FAGP method, ``c'' denotes the P2GP method, and ``d'' denotes the LINBSVM solver.}\label{table2}\end{minipage}},rotate=90,center}
\centering
\small
\begin{tabular}{|ccc|l|l|l|l|}
\hline
Data   & C & suppvec & \multicolumn{1}{c|}{$\text{R}_{\text{KKT}}$} & \multicolumn{1}{c|}{Time}  & \multicolumn{1}{c|}{Iter} &\multicolumn{1}{c|}{Accuracy (\%)}\\
 & & & \multicolumn{1}{c|}{a$|$b$|$c$|$d}  &  \multicolumn{1}{c|}{a$|$b$|$c$|$d}  & \multicolumn{1}{c|}{a$|$b$|$c$|$d} & \multicolumn{1}{c|}{a$|$b$|$c$|$d} \\
\hline
madelon &0.031 & 134 &  4.1-4 $|$ 9.9-4 $|$ 1.8-2 $|$ 1.2-2
&    01 $|$ 01 $|$ 01 $|$ 4:45
&   3(9) $|$ 133 $|$ 3(15,37) $|$ 20000
&   59.0 $|$ 59.0 $|$ 59.0 $|$ 58.6
\\\hline
ijcnn1 &1024 & 16410 &  9.6-4 $|$ 9.9-4 $|$ 1.7-4 $|$ 2.8-2
&    02 $|$ 45 $|$ 1:28 $|$ T
&   5(4) $|$ 3386 $|$ 38(1266,16454) $|$ 20000
&   91.5 $|$ 92.1 $|$ 91.4 $|$ 84.2
\\\hline
cod-rna &2 & 234 &  3.9-4 $|$ 6.9-4 $|$ 1.1-4 $|$ 1.6-0
&    02 $|$ 9 $|$ 75 $|$ T
&   6(2) $|$ 801 $|$ 17(400,8774) $|$ 1362
&   95.2 $|$ 95.2 $|$ 95.2 $|$ 0.0
\\\hline
svmguide1 &64 & 50 &  3.1-4 $|$ 4.7-4 $|$ 2.7-3 $|$ 6.9-4
&    00 $|$ 01 $|$ 07 $|$ 43
&   6(3) $|$ 301 $|$ 636(4352,10850) $|$ 1103
&   94.7 $|$ 94.8 $|$ 94.6 $|$ 94.6
\\\hline
w1a &1 & 126 &  2.6-4 $|$ 9.7-4 $|$ 1.1-3 $|$ 2.7-2
&   01 $|$ 01 $|$ 01 $|$ 6:04
&   6(7) $|$ 461 $|$ 5(30,287) $|$ 20000
&    97.7 $|$ 97.7 $|$ 97.7 $|$ 97.7
\\\hline
w2a &2 & 169 &  1.2-4 $|$ 9.8-4 $|$ 9.0-4 $|$ 1.5-2
&    01 $|$ 01 $|$ 01 $|$ 13:27
&    7(7) $|$ 673 $|$ 4(29,772) $|$ 20000
&   98.0 $|$ 98.0 $|$ 98.0 $|$ 98.0
\\\hline
w3a&2 & 194 &  4.9-4 $|$ 9.7-4 $|$ 7.8-4 $|$ 1.4-3
&    01 $|$ 02 $|$ 02 $|$ 31:57
&    7(7) $|$ 809 $|$ 7(41,2072) $|$ 20000
&    98.2 $|$ 98.2 $|$ 98.2 $|$ 98.2
\\\hline
w4a &1 & 226 &  7.8-4 $|$ 9.9-4 $|$ 1.6-3 $|$ 1.9-3
&    01 $|$ 03 $|$ 02 $|$ 1:18:32
&    5(10) $|$ 906 $|$ 5(29,1841) $|$ 20000
&    98.4 $|$ 98.4 $|$ 98.4 $|$ 98.4
\\\hline
w5a &1 & 229 &  2.1-4 $|$ 9.9-4 $|$ 1.2-3 $|$ 1.8-3
&    01 $|$ 05 $|$ 06 $|$ T
&    7(7) $|$ 1573 $|$ 8(44,3321) $|$ 15315
&    98.5 $|$ 98.5 $|$ 98.5 $|$ 98.5
\\\hline
w6a &1 & 360 &  5.0-4 $|$ 9.9-4 $|$ 1.8-3 $|$ 2.9-2
&    02 $|$ 16 $|$ 05 $|$ T
&    6(11) $|$ 2784 $|$ 4(30,2072) $|$ 4643
&    98.7 $|$ 98.7 $|$ 98.7 $|$ 98.7
\\\hline
w7a &2 & 601 &  8.7-4 $|$ 9.9-4 $|$ 1.0-3 $|$ 1.5-2
&    01 $|$ 02 $|$ 01 $|$ 33:10
&    5(10) $|$ 875 $|$ 4(37,1250) $|$ 20000
&    97.3 $|$ 97.2 $|$ 97.2 $|$ 97.3
\\\hline
w8a &32 & 5145 &   8.4-4 $|$ 9.9-4 $|$ 5.3-5 $|$ 1.5-1
&    05 $|$ 3:19 $|$ 1:03 $|$ T
&    10(8) $|$ 11528 $|$ 15(330,11069) $|$ 692
&    98.7 $|$ 98.7 $|$ 98.7 $|$ 87.7
\\\hline
mushrooms &0.25 & 376 &   6.2-4 $|$ 9.8-4 $|$ 3.2-2 $|$ 1.4-2
&    01 $|$ 2 $|$ 01 $|$ 1:02:25
&    7(7) $|$ 480 $|$ 2(63,149) $|$ 20000
&    100 $|$ 100 $|$ 100 $|$ 100
\\\hline
phishing  &0.063 & 220 &  6.3-4 $|$ 9.8-4 $|$ 2.9-1 $|$ 5.6-3
&    01 $|$ 2 $|$ 40 $|$ T
&    6(6) $|$ 520 $|$ 12(210,19888) $|$ 14644
&    94.1 $|$ 94.1 $|$ 89.6 $|$ 93.8
\\\hline
skin nonskin &8192 &156 &  5.7-4 $|$ 2.0-1 $|$ 3.7-3 $|$ 1.2-2
&    00 $|$ 00 $|$ 00 $|$ 01
&    5(2) $|$ 101 $|$ 1(9,17) $|$ 20000
&    83.1 $|$ 83.3 $|$ 82.1 $|$ 83.3
\\\hline
a2a  &0.25 &107 &  5.2-4 $|$ 9.8-4 $|$ 3.7-3 $|$ 2.5-3
&    00 $|$ 01 $|$ 01 $|$ 4:04
&    6(5) $|$ 407 $|$ 3(35,967) $|$ 20000
&    81.0 $|$ 80.9 $|$ 81.0 $|$ 81.9
\\\hline
a3a  &0.5 &102 &  7.1-4 $|$ 9.9-4 $|$ 2.3-3 $|$ 1.4-3
&    01 $|$ 01 $|$ 01 $|$ 9:29
&    5(5) $|$ 438 $|$ 3(37,819) $|$ 20000
&    84.0 $|$ 84.0 $|$ 84.0 $|$ 83.6
\\\hline
a4a  &256 &191 &  4.5-4 $|$ 9.9-4 $|$ 1.1-4 $|$ 1.5-2
&    01 $|$ 06 $|$ 08 $|$ 24:50
&    6(9) $|$ 4638 $|$ 25(304,19667) $|$ 20000
&    83.9 $|$ 83.9 $|$ 83.5 $|$ 57.7
\\\hline
a6a  &0.03 &178 &  5.9-4 $|$ 9.7-4 $|$ 2.3-2 $|$ 5.2-3
&    01 $|$ 01 $|$ 04 $|$ T
&    4(8) $|$ 456 $|$ 5(120,2722) $|$ 15893
&    84.2 $|$ 84.2 $|$ 84.2 $|$ 84
\\\hline
a7a  &0.5 &102&  5.3-4 $|$ 9.8-4 $|$ 2.5-3 $|$ 4.5-3
&    01 $|$ 04 $|$ 02 $|$ T
&    6(6) $|$ 1142 $|$ 1(42,1086) $|$ 9601
&    84.6 $|$ 84.6 $|$ 84.6 $|$ 84
\\\hline
a8a  &2 &182 &  5.8-4 $|$ 9.8-4 $|$ 6.8-4 $|$ 7.4-1
&    01 $|$ 10 $|$ 07 $|$ T
&    5(8) $|$ 2228 $|$ 3(77,2904) $|$ 5229
&    84.5 $|$ 84.5 $|$ 84.6 $|$ 67.6
\\\hline
a9a  &1 &190 &  4.2-4 $|$ 9.8-4 $|$ 1.3-3 $|$ 1.4-0
&    01 $|$ 15 $|$ 06 $|$ T
&    5(9) $|$ 2318 $|$ 2(59,2056) $|$ 2318
&    84.8 $|$ 84.8 $|$ 84.8 $|$ 67.6
\\\hline
\end{tabular}
\end{adjustbox}
\end{table}

In Table \ref{table2}, we present the performances of the four algorithms on various C-SVC problems \eqref{Def_Dual_of_C_SVM_Problem} with the linear kernel function. Note that our SSsNAL method not only achieves the desired accuracy in all cases, but also outperforms the other three algorithms. For example, in the `w8a' data set, only the SSsNAL and FAPG methods achieve the predetermined accuracy, but the SSsNAL method only takes 5 seconds to reach the desired accuracy while the FAPG method needs more than three minutes and the other two solvers needs much more time.
The test accuracy of the classification on most of problems are nearly the same for different algorithms except for the LINBSVM. For the datasets `jcnn1', `a4a', `a8a' and `a9a', the accuracies of the LINBSVM are worse than those of the other three algorithms, because its KKT residuals obviously fail to achieve the required accuracy. Hence, to achieve a satisfactory classification, an appropriate precision requirement is necessary.

\subsection{Numerical experiments for the Nystr\"{o}m method and the random Fourier features method to the RBF kernel support vector classification problems}

In this subsection, we compare the performances of the SSsNAL, FAPG and P2PG methods to the approximated linear kernel support vector classification problem \eqref{Def_Dual_of_C_SVM_Problem} by the Nystr\"{o}m method and the random Fourier features method.

First, we give a brief introduction to the Nystr\"{o}m method \cite{williams2001using}. Recall that $(\tilde{\vx}_{i},y_{i}),\,i = 1, . . . , n$, denote a set of training examples, where $\tilde{\vx}_{i} \in \mathbb{R}^{p}$ and $y_{i} \in \{+1, -1\}$. Let $K\in \mathbb{R}^{n\times n}$ be the RBF kernel matrix with $K_{ij}=\exp^{-||\tilde{\vx}_{i}-\tilde{\vx}_{j}||^2/2\alpha}\ (i,j=1,\ldots,n$), where $\alpha>0$ is a given parameter. Instead of computing the whole kernel matrix directly, the Nystr\"{o}m method provides a reduced-rank approximation to the kernel matrix $K$ by choosing a landmark points set $\mathcal{I}$ consisting of the column indicators of matrix $K$ and then setting
$$\widetilde{K}=K_{n,r}K_{r,r}^{-1}K_{n,r}^{T},$$
where $K_{n,r}\in \mathbb{R}^{n\times r}$ is a submatrix consisting of the columns of $K$ indexed
by $\mathcal{I}$, $K_{r,r} \in \mathbb{R}^{r\times r}$ is also a submatrix of $K$ with the rows and columns in $\mathcal{I}$ and $r$ is the cardinality of the set $\mathcal{I}$. Next, we can use the reduced-rank approximation matrix $\tilde{K}$ to replace the original kernel matrix $K$ in the problem \eqref{Def_Dual_of_C_SVM_Problem} and turn the problem into a linear kernel support vector classification problem.

Next, we introduce the random Fourier features method \cite{rahimi2008random} which is another commonly used randomized algorithm for approximating the kernel matrix. Let $k(\vx,\vy)=\exp^{-||\vx-\vy||^2/2\alpha}$ be the RBF kernel function. Then, according to the Bochner's theorem \cite{rudin1962fourier}, it can be regarded as the Fourier transform of the Gaussian probability density function $p(\omega)=(2\pi/\alpha^n)^{-\frac{n}{2}}\exp^{-\alpha||\omega||^2/2}$, i.e.,
    $$k(\vx,\vy)=\int_{\mathbb{R}^q}\exp^{i\omega^{T}(\vx-\vy)}p(\omega)
    =E_{\omega}[\vz_{\omega}(\vx)^{T}\vz_{\omega}(\vy)],$$
    where the function $\vz_{\omega}(\vx):=[\cos(\omega^{T}\vx), \sin(\omega^{T}\vx)]^{T}\in \mathbb{R}^{2}$ is a real-valued mapping. So, the random Fourier features method is constructed by first sampling the independent and identically distributed vectors $\omega_{1},\ldots,\omega_{N}\in \mathbb{R}^{q}$ from $p(\omega)$ and then using the sample average $\frac{1}{N}\sum_{i=1}^{N}\vz_{\omega_{i}}(\vx)^{T}\vz_{\omega_{i}}(\vy)$ to estimate $k(\vx,\vy)$. If $K\in \mathbb{R}^{n\times n}$ is the RBF kernel matrix, then the random Fourier features method is to approximate $K$ by
    $$\widetilde{K}=Z_{\omega}Z_{\omega}^{T},$$
    where
$$ Z_{\omega}:=\left(
      \begin{array}{ccccc}
        \cos(\omega^{T}_{1}\tilde{\vx}_{1}) & \sin(\omega^{T}_{1}\tilde{\vx}_{1}) & \cdots & \cos(\omega^{T}_{N}\tilde{\vx}_{1}) & \sin(\omega^{T}_{N}\tilde{\vx}_{1}) \\
        \vdots & \vdots & \ddots& \vdots & \vdots \\
        \cos(\omega^{T}_{1}\tilde{\vx}_{n}) & \sin(\omega^{T}_{1}\tilde{\vx}_{n}) & \cdots & \cos(\omega^{T}_{N}\tilde{\vx}_{n}) & \sin(\omega^{T}_{N}\tilde{\vx}_{n}) \\
      \end{array}
    \right)\in \mathbb{R}^{n\times 2N},
    $$
and $\tilde{\vx}_{i} \in \mathbb{R}^{q}$ $(i=1,\ldots,n)$ are training samples. Similarly, we use the matrix $\widetilde{K}$ to replace the original kernel matrix $K$ in the problem \eqref{Def_Dual_of_C_SVM_Problem} and turn the problem into a linear kernel support vector classification problem with $Z_{\omega}$ as new training samples.

In the implementation of the Nystr\"{o}m method, we replace $K_{r,r}$ by $K_{r,r}+\sigma I$ with $\sigma=10^{-3}$  \cite{Neal1998regression} to avoid numerical instabilities. Moreover, we chose the $k$-means clustering algorithm \cite{zhang2008improved} to determine the landmark points set $\mathcal{I}$ and the cardinality of $\mathcal{I}$ was given by $r=\{128,256,512,1024\}$. Similarly, in the implementation of the random Fourier features method, we set the sampling number to be $2N=\{128,256,512,1024\}$.
As for the parameters $C$ and $\alpha$, they are also selected by the ten-fold cross validation. After approximating the kernel matrix by the Nystr\"{o}m method and the random Fourier features method, we apply the SSsNAL, FAPG and P2PG methods, respectively, to solve the approximated linear C-SVC problems. Moreover, the termination criterion of the three algorithms are the same as \eqref{Def_Rkkt} except for replacing $Q$ with $\widehat{Q}$, where $\widehat{Q}_{ij}=y_i y_j \widehat{K}_{ij},\,i,j=1,\ldots,n$.

\begin{table}[htpb]
\begin{adjustbox}{addcode={\begin{minipage}{\width}}{\caption{Comparisons of the SSsNAL, FAPG and P2GP methods for the standard C-SVM problem \eqref{Def_Dual_of_C_SVM_Problem} with the RBF kernel function by the Nystr\"{o}m method. In the table, ``a'' denotes the SSsNAL method, ``b'' denotes the FAGP method and ``c'' denotes the P2PG method. In the last column, ``Rs'' denotes the reference number of unbounded support vectors and ``Ra'' denotes the reference classification accuracy on the testing data set (\%).}\label{table3_1}\end{minipage}},rotate=0,center}
\centering
\small
\begin{tabular}{|c|l|c|c|c|l|l|c|}
\hline
Data   & \multicolumn{1}{c|}{r} & suppvec & \multicolumn{1}{c|}{$\text{R}_{\text{KKT}}$} & \multicolumn{1}{c|}{Time}  & \multicolumn{1}{c|}{Iter} &\multicolumn{1}{c|}{Accuracy (\%)} &\multirow{2}{*}{\shortstack{Original\\Rs$|$Ra}}  \\
C:$\alpha$ & & & \multicolumn{1}{c|}{a$|$b$|$c}  &  \multicolumn{1}{c|}{a$|$b$|$c}  & \multicolumn{1}{c|}{a$|$b$|$c} & \multicolumn{1}{c|}{a$|$b$|$c} &  \\
\hline
\multirow{5}{*}{\shortstack{svmguide1\\0.063:1}}
     & 128   & 35    & 8.0-4 $|$ 9.5-4 $|$ 1.3-4 & 01 $|$ 01 $|$ 01   & 5(5) $|$ 403  $|$ 1(6,100) & 73.7 $|$ 73.7 $|$ 74.0 & \multirow{5}{*}{\shortstack{36 $|$ 96.3}}   \\
           & 256   & 44    & 8.8-4 $|$ 1.0-3 $|$ 1.1-4 & 00   $|$ 01   $|$ 00   & 6(4) $|$ 403 $|$ 1(6,72)& 76.7 $|$ 76.5 $|$ 76.5 & \\
           & 512   & 44    & 6.9-4 $|$ 9.9-4 $|$ 8.7-4 & 00   $|$ 01   $|$ 00   & 5(5) $|$ 372  $|$ 1(6,129) & 77.5 $|$ 77.6 $|$ 77.1 & \\
           & 1024  & 44    & 3.1-4 $|$ 9.9-4 $|$ 1.1-4 & 01   $|$ 02   $|$ 01   & 6(4) $|$ 376  $|$ 1(6,107) & 76.9 $|$ 77.0 $|$ 76.7 & \\
    \hline
    \multirow{5}{*}{\shortstack{w5a \\ 256:256}}
      & 128   & 14    & 9.6-4 $|$ 1.0-3 $|$ 1.2-4 & 00  $|$ 03 $|$ 01  & 5(3) $|$ 460 $|$ 1(5,173)  & 90.8 $|$ 90.8 $|$ 97.0  & \multirow{5}{*}{\shortstack{1390 $|$ 98.5}}  \\
           & 256   & 15    & 8.9-4 $|$ 9.9-4 $|$ 2.8-4 & 00  $|$ 03  $|$ 01  & 6(3) $|$ 376 $|$ 1(5,85)  & 96.9 $|$ 90.1 $|$ 89.7  & \\
          & 512   & 20    & 8.9-4 $|$ 9.5-4 $|$ 1.5-4 & 01  $|$ 05  $|$ 03   & 6(3) $|$ 372 $|$ 4(20,261) & 97.0 $|$ 97.0 $|$ 97.0  & \\
          & 1024  & 16    & 8.4-4 $|$ 9.7-4 $|$ 2.7-4 & 01  $|$ 07  $|$ 02   & 6(3) $|$ 360 $|$ 1(5,96) & 97.3 $|$ 96.8 $|$ 97.0  & \\
    \hline
    \multirow{5}{*}{\shortstack{w6a \\ 16:16}}
     & 128   & 7     & 8.4-4 $|$ 9.0-4 $|$ 1.0-4 & 01 $|$ 05  $|$ 01   &9(2) $|$ 592  $|$ 2(26,149) & 64.8 $|$ 64.8 $|$ 64.6 & \multirow{5}{*}{\shortstack{2613 $|$ 98.8}} \\
           & 256   & 9     & 4.0-4 $|$ 9.9-4 $|$ 4.5-4 & 01  $|$ 07  $|$  01   &9(2) $|$ 573  $|$ 1(25,80) & 93.6 $|$ 93.6 $|$ 94.6  &\\
           & 512   & 13    & 7.7-4 $|$ 9.9-4 $|$ 5.5-4 & 01   $|$ 11 $|$ 02   &8(3) $|$ 561  $|$ 1(24,85)  & 93.3 $|$ 93.3 $|$ 93.4 & \\
           & 1024  & 15    & 6.9-4 $|$ 9.9-4 $|$ 5.6-4 & 03  $|$ 17  $|$ 03   &8(2) $|$ 536  $|$ 1(24,70) & 94.9 $|$ 94.8 $|$ 94.9  &\\
    \hline
    \multirow{5}{*}{\shortstack{phishing \\ 2:4}}
    & 128   & 42    & 8.2-4 $|$ 8.3-4 $|$ 1.5-4 & 00   $|$ 01   $|$ 15  & 6(3) $|$ 263 $|$ 10(162,7188) & 79.2 $|$ 76.7 $|$ 76.7 & \multirow{5}{*}{\shortstack{3528 $|$ 97.4}} \\
           & 256   & 43    & 5.2-4 $|$ 7.1-4 $|$ 1.3-4 & 01  $|$ 01  $|$ 34  & 6(4) $|$ 268 $|$ 59(487,6628) & 83.8 $|$ 82.2 $|$ 75.2  & \\
          & 512   & 60    & 6.6-4 $|$ 9.1-4 $|$ 1.5-4 & 01   $|$ 02   $|$ 01   & 6(4) $|$ 254 $|$ 3(18,128) & 83.6 $|$ 85.3 $|$ 89.4 & \\
          & 1024  & 76    & 6.2-4 $|$ 9.3-4 $|$ 1.7-4 & 02   $|$ 03   $|$ 01   & 5(5) $|$ 255 $|$ 3(60,42) & 74.7 $|$ 77.1 $|$ 75.5 & \\
    \hline
    \multirow{5}{*}{\shortstack{a6a \\ 32:16}}
     & 128   & 43    & 7.2-4 $|$ 9.9-4 $|$ 1.6-4 & 00  $|$  03  $|$ 01   & 9(2) $|$ 551 $|$ 1(5,403) & 71.3 $|$ 60.3 $|$ 77.1 & \multirow{5}{*}{\shortstack{681 $|$ 84.6}} \\
          & 256   & 46    & 7.9-4 $|$ 9.5-4 $|$ 1.6-4 & 00 $|$ 04   $|$ 02   & 5(3) $|$ 535 $|$ 1(6,421) & 41.3 $|$ 42.0 $|$ 34.5 & \\
          & 512   & 66    & 7.8-4 $|$ 9.9-4 $|$ 1.7-4 & 01  $|$ 06   $|$ 03   & 5(4) $|$ 532 $|$ 1(5,358) & 54.1 $|$ 59.8 $|$ 69.7 & \\
          & 1024  & 75    & 6.9-4 $|$ 9.9-4 $|$ 1.9-4 & 02  $|$ 10  $|$ 06   & 5(4) $|$ 528 $|$ 1(7,367) & 41.3 $|$ 34.5 $|$ 32.3 & \\
    \hline
    \multirow{5}{*}{\shortstack{a7a \\ 32:4}}
      &  128     & 48    & 7.1-4 $|$ 9.9-4 $|$ 6.6-4 & 01 $|$ 05 $|$ 02  & 7(4) $|$ 647 $|$ 1(5,468) & 58.6 $|$ 58.6 $|$ 58.6 & \multirow{5}{*}{\shortstack{460 $|$ 84.9}} \\
          & 256      & 46    & 8.1-4 $|$ 9.9-4 $|$ 6.6-4 & 01  $|$ 07  $|$ 03  & 6(4) $|$ 640 $|$ 1(5,477) & 59.7 $|$ 58.6 $|$ 75.7 & \\
          & 512      & 72    & 7.3-4 $|$ 9.8-4 $|$ 6.4-4 & 01  $|$ 10  $|$ 05  & 6(5) $|$ 625 $|$ 1(5,363)  & 54.8 $|$ 58.6 $|$ 70.3 & \\
          & 1024      & 78    & 6.4-4 $|$ 9.9-4 $|$ 6.8-4 & 02   $|$ 17  $|$ 14 & 6(5) $|$ 602 $|$ 3(17,540) & 70.6 $|$ 72.6 $|$ 75.7 & \\
    \hline
\end{tabular}
\end{adjustbox}
\end{table}

Tables \ref{table3_1} and \ref{table3_2} show the performances of the SSsNAL, FAPG and P2PG methods on several approximated linearization RBF kernel C-SVC problems \eqref{Def_Dual_of_C_SVM_Problem} obtained by the Nystr\"{o}m method and the random Fourier features method, respectively. For comparisons, we also report the number of the unbounded support vectors and the classification accuracy of the SSsNAL method on the original RBF kernel C-SVC problems in the last column of the two tables. Above all, similar to the previous subsection, the SSsNAL method takes less time to solve the approximation problems than the FAPG and P2PG methods with the same stopping criterion. Indeed, as we see in the next subsection, for all of the three algorithms the computing times for solving the approximation problems are significantly less than those for solving the original problems. However, the test accuracies for solving the original problems are more ideal. In addition, the numbers of the unbounded support vectors obtained by solving the approximation problems are also quite different from the true numbers. As is explained in \cite{nalepa2019selecting}, if the Nystr\"{o}m and the random Fourier features methods are used to reduce the size of training sets by selecting the candidate vectors (i.e., the support vectors), it may produce large errors. This also shows the necessity and importance of solving the original problem.

Although the Nystr\"{o}m method and the random Fourier features method have some disadvantages, they are very suitable to be implemented to generate an initial iteration point. Therefore, in the following subsection we adopt the random Fourier features method with $2N=1024$ samples to warm start the SSsNAL method.

\begin{table}[htpb]
\begin{adjustbox}{addcode={\begin{minipage}{\width}}{\caption{Comparisons of the SSsNAL, FAPG and P2GP methods for the standard C-SVM problem \eqref{Def_Dual_of_C_SVM_Problem} with the RBF kernel function by the random Fourier features method. In the table, ``a'' denotes the SSsNAL method, ``b'' denotes the FAGP method and ``c'' denotes the P2PG method. In the last column, ``Rs'' denotes the reference number of unbounded support vectors and ``Ra'' denotes the reference classification accuracy on the testing data set (\%).}\label{table3_2}\end{minipage}},rotate=0,center}
\centering
\small
\begin{tabular}{|c|l|c|c|c|l|l|c|}
\hline
Data   & \multicolumn{1}{c|}{2N} & suppvec & \multicolumn{1}{c|}{$\text{R}_{\text{KKT}}$} & \multicolumn{1}{c|}{Time}  & \multicolumn{1}{c|}{Iter} &\multicolumn{1}{c|}{Accuracy (\%)} &\multirow{2}{*}{\shortstack{Original\\Rs$|$Ra}}  \\
C:$\alpha$ & & & \multicolumn{1}{c|}{a$|$b$|$c}  &  \multicolumn{1}{c|}{a$|$b$|$c}  & \multicolumn{1}{c|}{a$|$b$|$c} & \multicolumn{1}{c|}{a$|$b$|$c}  & \\
\hline
\multirow{4}{*}{\shortstack{svmguide1\\0.063:1}}
     & 128   & 40    & 3.1-4 $|$ 1.0-3 $|$ 1.5-4 & 01 $|$ 01 $|$ 01 & 3(9) $|$ 296   $|$ 2(55,150) & 96.5 $|$ 96.6 $|$ 96.6 & \multirow{4}{*}{36 $|$ 96.3}\\
          & 256   & 38    & 8.2-4 $|$ 9.5-4 $|$ 2.2-4 & 01 $|$ 00 $|$ 00 & 4(8)  $|$ 290  $|$ 2(55,104) & 94.8 $|$ 95.0 $|$ 95.1 & \\
          & 512   & 35    & 2.2-4 $|$ 9.9-4 $|$ 1.9-4 & 00 $|$ 10 $|$ 00 & 4(7)  $|$ 294 $|$ 3(60,97) & 94.8 $|$ 94.7 $|$ 94.5 & \\
          & 1024  & 39    & 4.2-4 $|$ 1.0-3 $|$ 2.1-4 & 01 $|$ 01 $|$ 01 & 4(7) $|$ 314 $|$ 2(55,92) & 95.3 $|$ 95.3 $|$ 95.4 & \\
    \hline
    \multirow{5}{*}{\shortstack{w5a \\ 256:256}}
       & 128   & 904   & 5.4-4 $|$ 5.4-4 $|$ 4.7-4 & 00 $|$ 01 $|$ 10 & 8(2) $|$ 301 $|$ 10(50,3965) & 54.1 $|$ 52.8 $|$ 49.4 & \multirow{4}{*}{1390 $|$ 98.5} \\
          & 256   & 441   & 7.2-4 $|$ 5.0-4 $|$ 3.4-4 & 01 $|$ 01 $|$ 07 & 10(2) $|$ 301  $|$ 5(25,1555) & 51.0 $|$ 50.8 $|$ 55.1 & \\
          & 512   & 492   & 6.5-4 $|$ 4.8-4 $|$ 6.9-4 & 01 $|$ 02 $|$ 04 & 10(2) $|$ 301   $|$ 3(15,655) & 64.7 $|$ 62.7 $|$ 58.1 & \\
          & 1024  & 368   & 8.1-4 $|$ 4.0-4 $|$ 5.0-4 & 03 $|$ 03 $|$ 07 & 12(2) $|$ 301 $|$ 4(20,589)  & 88.8 $|$ 87.4 $|$ 89.4 & \\
    \hline
    \multirow{5}{*}{\shortstack{w6a \\ 16:16}}
     & 128   & 460   & 6.2-4 $|$ 9.8-4 $|$ 7.8-4 & 01 $|$ 06   $|$ 16    & 6(4) $|$ 1296  $|$ 4(122,4239) & 50.9 $|$ 47.6 $|$ 43.0 & \multirow{4}{*}{2613 $|$ 98.8} \\
          & 256   & 454   & 7.5-4 $|$ 1.0-3 $|$ 1.1-4 & 01 $|$ 03    $|$ 10    & 6(4) $|$ 501  $|$ 5(25,542) & 61.0 $|$ 62.4 $|$ 63.0 & \\
          & 512   & 2031  & 2.4-4 $|$ 1.0-3 $|$ 1.2-4 & 03    $|$ 05   $|$ 06     & 6(6) $|$ 485  $|$ 5(25,542) & 88.6 $|$ 71.0 $|$ 69.2 & \\
          & 1024  & 2139  & 4.9-4 $|$ 1.0-3 $|$ 1.3-4 & 08  $|$ 10   $|$ 12    & 6(5) $|$ 495  $|$ 4(20,517) & 92.9 $|$ 92.8 $|$ 92.6 & \\
    \hline
    \multirow{5}{*}{\shortstack{phishing \\ 2:4}}
    & 128   & 1053 & 6.4-4 $|$ 4.8-4 $|$ 1.9-4 & 00    $|$ 01  $|$ 38    & 4(5) $|$ 434 $|$ 6(144,20952) & 92.5 $|$ 92.4 $|$ 92.4 & \multirow{4}{*}{3528 $|$ 97.4} \\
          & 256   & 466 & 3.8-4 $|$ 5.0-4 $|$ 1.9-4 & 01   $|$ 01   $|$ 64    & 5(5) $|$ 301  $|$ 13(389,22091) & 92.8 $|$ 92.7 $|$ 93.8 & \\
          & 512   & 554 & 4.1-4 $|$ 6.4-4 $|$ 1.9-4 & 02   $|$  02   $|$ 47    & 6(7) $|$ 281  $|$ 8(172,9024) & 92.7 $|$ 92.7 $|$ 92.7 & \\
          & 1024  & 690 & 2.9-4 $|$ 8.3-4 $|$ 1.6-4 & 03  $|$ 02 $|$ 21    & 6(9) $|$ 247  $|$ 5(69,2049) & 94.3 $|$ 94.3 $|$ 94.3 & \\
    \hline
    \multirow{5}{*}{\shortstack{a6a \\ 32:16}}
     & 128   & 1069 & 8.3-4 $|$ 5.0-4 $|$ 6.9-4 & 00  $|$ 01    $|$ 01     & 6(2) $|$ 434 $|$ 1(50,744) & 73.6 $|$ 73.5 $|$ 74.1 & \multirow{4}{*}{681 $|$ 84.6} \\
          & 256   & 443   & 6.6-4 $|$ 3.2-4 $|$ 6.9-4 & 00   $|$ 01  $|$ 02  & 6(2) $|$ 368 $|$ 1(43,710) & 78.7 $|$ 78.8 $|$ 78.5 & \\
          & 512   & 538   & 5.2-4 $|$ 3.6-4 $|$ 6.9-4 & 01   $|$ 02   $|$ 04     & 7(2) $|$ 401 $|$ 2(45,601) & 79.5 $|$ 79.6 $|$ 79.7 & \\
          & 1024  & 729 & 5.4-4 $|$ 3.3-4 $|$ 7.7-4 & 02   $|$ 03 $|$ 08     & 6(2) $|$ 301  $|$ 2(43,681) & 82.1 $|$ 82.1 $|$ 82.1 & \\
    \hline
    \multirow{5}{*}{\shortstack{a7a \\ 32:4}}
     & 128   & 409 &6.9-4 $|$ 5.8-4 $|$ 1.2-4 & 00   $|$ 01  $|$ 15    & 4(3) $|$ 301  $|$ 3(129,5375) & 76.8 $|$ 76.6 $|$ 76.0 & \multirow{4}{*}{460 $|$ 84.9} \\
          & 256   & 438 & 7.5-4 $|$ 5.2-4 $|$ 1.2-4 & 01   $|$ 02  $|$ 02     & 5(3) $|$ 301   $|$ 1(48,361) & 80.1 $|$ 80.1 $|$ 80.0 & \\
          & 512   & 550 & 3.7-4 $|$ 5.2-4 $|$ 1.2-4 & 01  $|$ 02   $|$ 03     & 6(3) $|$ 301   $|$ 1(26,315) & 82.8 $|$ 82.9 $|$ 82.9 &\\
          & 1024  & 747 & 4.8-4 $|$ 5.1-4 $|$ 1.5-4 & 03  $|$ 04  $|$ 05     & 5(4) $|$ 301 $|$ 2(55,280) & 83.5 $|$ 83.6 $|$ 83.5 &\\
    \hline
\end{tabular}
\end{adjustbox}
\end{table}

\subsection{Numerical experiments for the support vector classification problems with the RBF kernel}\label{RBF_kernel_case}

In this subsection, we present the numerical experiments of all the algorithms on the C-SVC problem \eqref{Def_Dual_of_C_SVM_Problem} with the RBF kernel function.

\begin{table}[htpb]
\begin{adjustbox}{addcode={\begin{minipage}{\width}}{\caption{Comparisons of the SSsNAL, FAPG, P2GP, and LINBSVM solvers for the standard C-SVM problems \eqref{Def_Dual_of_C_SVM_Problem} with the RBF kernel function. In the table, ``a'' denotes the SSsNAL method, ``b'' denotes the FAGP method, ``c'' denotes the P2GP method, and ``d'' denotes the LINBSVM solver.}\label{table3}\end{minipage}},rotate=90,center}
\centering
\small
\begin{tabular}{|cll|l|l|l|l|}
\hline
Data   & C:$\alpha$ & suppvec & \multicolumn{1}{c|}{$\text{R}_{\text{KKT}}$} & \multicolumn{1}{c|}{Time}  & \multicolumn{1}{c|}{Iter} &\multicolumn{1}{c|}{Accuracy (\%)}\\
& & & \multicolumn{1}{c|}{a$|$b$|$c$|$d}  &  \multicolumn{1}{c|}{a$|$b$|$c$|$d}  & \multicolumn{1}{c|}{a$|$b$|$c$|$d} & \multicolumn{1}{c|}{a$|$b$|$c$|$d} \\
    \hline
    svmguide1 & 0.063:1 & 30    & 4.1-4 $|$ 9.6-4 $|$ 9.4-4 $|$ 9.2-1 & 1  $|$ 6 $|$ 3 $|$ 28 & 7(2) $|$ 327  $|$ 3(60,113) $|$ 728 & 96.3 $|$ 96.3 $|$  96.3 $|$ 84.1 \\ \hline
    w1a   & 64:64 & 474   & 7.8-4 $|$ 9.9-4 $|$ 3.0-5 $|$ 5.3-2 & 6 $|$ 9 $|$ 9 $|$ 4:17 & 6(3) $|$ 234 $|$ 5(38,265) $|$ 20000 & 97.8 $|$ 97.8 $|$ 97.8 $|$ 81.1 \\ \hline
    w2a   & 16:16 & 713   & 5.6-4 $|$ 9.9-4 $|$ 9.9-5 $|$ 2.5-1 & 11 $|$ 19 $|$ 13 $|$ 8:06 & 8(6) $|$ 333 $|$ 4(58,187) $|$ 20000 & 98.1 $|$ 98.1 $|$ 98.1 $|$ 56.7 \\ \hline
    w3a   & 32:128 & 889   & 6.0-4 $|$ 9.9-4 $|$ 1.7-5 $|$ 3.5-2 & 22 $|$ 32 $|$ 27 $|$ 16:24 & 9(2) $|$ 378 $|$ 4(52,443) $|$ 20000 & 98.1 $|$ 98.1 $|$ 98.2 $|$ 50.9 \\ \hline
    w4a   & 32:128 & 1193  & 6.9-4 $|$ 9.9-4 $|$ 1.8-3 $|$ 4.6-2 & 37 $|$ 4:11 $|$ T  $|$ 40:46 & 9(3) $|$ 482  $|$ 2(12,182) $|$ 20000 & 98.4 $|$ 98.4 $|$ 98.4 $|$ 46.9 \\ \hline
    w5a   & 256:256 & 548  & 9.4-4 $|$ 7.2-4 $|$ 5.3-3 $|$ 1.3-2 & 31 $|$ 29:46 $|$ T $|$ 1:10:11 & 9(1) $|$301 $|$ 1(5,73) $|$ 20000 & 98.5 $|$ 98.5 $|$ 98.3 $|$ 75.2 \\ \hline
    w6a   & 16:16 & 2668  & 2.9-4 $|$ 2.5-2 $|$ 7.5-1 $|$ 9.3-1 & 1:03:04 $|$ T $|$ T $|$ T & 13(3) $|$ 157   $|$ 1(5,30) $|$ 8142 & 98.8 $|$ 98.8 $|$ 42.8 $|$ 37.6 \\ \hline
    mushrooms & 1:0.063 & 4040 & 4.7-4 $|$ 5.8-4 $|$ 1.5-2 $|$ 7.2-1 & 2:07  $|$ 11:13 $|$ 5:26 $|$ 44:31 & 8(3) $|$ 150 $|$ 1(5,23) $|$ 20000 & 99.8 $|$ 99.8 $|$ 99.8 $|$ 92.4  \\ \hline
    phishing & 2:4   & 3529  & 5.3-4 $|$ 9.8-4 $|$ 1.0-3 $|$ 1.2-0 & 1:38 $|$ 5:19 $|$ 32:00 $|$ 1:21:37 & 9(2) $|$ 195 $|$ 2(18,122) $|$ 20000 & 97.4 $|$ 97.4 $|$ 97.4 $|$ 93.4  \\ \hline
    a2a   & 128:512 & 398 & 6.9-4 $|$ 9.2-4 $|$ 2.6-6 $|$ 2.4-2 & 2 $|$ 4 $|$ 8 $|$ 2:07 & 4(1) $|$ 201  $|$ 4(59,1584) $|$ 20000  & 82.1 $|$ 81.8 $|$ 82.3 $|$ 71.4\\ \hline
    a3a   & 128:32 & 157 & 7.3-4 $|$ 6.9-4 $|$ 3.1-5 $|$ 4.4-1 & 3 $|$ 4 $|$ 3 $|$ 4:19 & 4(1) $|$ 121  $|$ 2(41,149) $|$ 20000 & 83.3 $|$ 83.3 $|$ 83.3 $|$ 70.1\\ \hline
    a4a   & 32:4  & 148 & 5.7-4 $|$ 3.2-4 $|$ 2.7-4 $|$ 1.3-0 & 6  $|$ 13 $|$ 7 $|$ 9:48 & 7(2) $|$ 195 $|$ 2(45,105) $|$ 20000 & 84.5 $|$ 84.5 $|$ 84.5 $|$ 72.0 \\ \hline
    a6a   & 32:16 & 628 & 6.2-4 $|$ 4.4-4 $|$ 3.8-2 $|$ 1.6.0 & 1:31  $|$ 18:30 $|$ T $|$ 1:11:27 & 8(2) $|$ 332 $|$ 1(45,194) $|$ 20000 & 84.4 $|$ 84.3 $|$ 83.1 $|$ 71.3  \\ \hline
    a7a   & 32:4  & 405 & 5.2-4 $|$ 5.2-4 $|$ 2.9-1 $|$ 1.2-0 & 7:17  $|$ 1:11:43 $|$ T $|$ T & 8(2) $|$ 301 $|$ 1(39,62) $|$ 16871 & 84.6 $|$ 84.7 $|$ 81.8 $|$ 71.4 \\ \hline
\end{tabular}
\end{adjustbox}
\end{table}

In Table \ref{table3}, we present the performances of the four algorithms on the various problems. In order to avoid saving the $n \times n$ dense kernel matrix $Q$, which may cause an out-of-memory exception, we only save partial columns of $Q$, denoted by $Q_{\mathcal{P}}$. Note that $Q_{\mathcal{P}} \in \mathbb{R}^{n \times p}$, where $p=\min\{n,\lfloor 3.6\times 10^{7}/n \rfloor\}$. So the number of entries in $Q_{\mathcal{P}}$ is less than $6000^2$. Then, we compute the rest entries of $Q$ on demand in each iteration of the FAGP method and the P2GP method. We can see the LIBSVM solver fails to produce a reasonably accurate solution for all the problems and its accuracies of the classification are worse than those of the other three solvers. Meanwhile, the SSsNAL, FAGP and P2GP methods can solve all the problems successfully. Nevertheless, in most cases, the FAGP and P2GP methods require much more time than the SSsNAL method. One can also observe that for large-scale problems, the SSsNAL method outperforms the FAGP and P2GP methods by a large margin (sometimes up to a factor of $10\sim100$). In fact, in almost all cases, when the iteration point is near the solution point, the number of unbounded support vectors is much smaller than the number of training data. This implies that the subproblems in our SSsNAL method can be solved by the semismooth Newton method very efficiently because the Newton direction can be obtained just by solving a much smaller linear system. The superior numerical performance of the SSsNAL method shows that it is a robust, high-performance algorithm for large-scale SVM problems.

\subsection{Numerical experiments for the support vector regression problems with the RBF kernel}\label{RBF_kernel_regression}
Three experiments based on the benchmark datasets from the LIBSVM data are performed to compare the SSsNAL method with the FAPG and P2PG methods on the SVR problems \eqref{Def_Dual_of_SVR_Problem} with the RBF kernel. In these experiments, we also randomly partition each dataset $\{(\tilde{\vx}_{i},y_{i})\in \mathbb{R}^{q}\times \mathbb{R}\ ,i = 1, . . . , n\}$ into a training set $\mathcal{D}_{\text{tr}}$ and a testing set $\mathcal{D}_{\text{te}}$ in the proportion 8:2 and normalize each variable $(\tilde{\vx}_{i},y_{i})$ into  $[0,1]^{q+1}$. Additionally, the spread parameter $\alpha$ in the Gaussian kernel $k(\vx,\vy)=\exp^{-||\vx-\vy||^2/2\alpha}$ and the regularization parameter C are selected by the ten-fold cross validation, while the insensitivity parameter $\varepsilon$ is determined directly by the method proposed in \cite{cherkassky2004practical}. The mean square error, defined by
$$\textrm{MSE} = \frac{1}{n_{\text{te}}}\sum_{i\in \mathcal{D}_{\text{te}}}
\left(y_i-\vk_{\tilde{\vx}_{i}}^{T}\vw^{*}-b^{*}\right)^2,$$
is used to characterize the prediction accuracy on the testing set, where $(\vw^{*},b^{*})\in \mathbb{R}^{n_{\text{tr}}}\times \mathbb{R}$ is the optimal solution of the SVR problem \eqref{Def_SVR_Problem}, $\vk_{\tilde{\vx}_{i}}=[k(\tilde{\vx}_{j_1},\tilde{\vx}_{i}),\ldots,k(\tilde{\vx}_{j_{n_{\text{tr}}}},\tilde{\vx}_{i})]^{T}\in \mathbb{R}^{n_{\text{tr}}}$ with $\tilde{\vx}_{j_k}\in \mathcal{D}_{\text{tr}}, k = 1,\ldots,n_{\text{tr}},$ is a given vector and $n_{\text{tr}}$ and $n_{\text{te}}$ are the cardinality of the training set and a testing set, respectively.
Considering that solving the SVR problems usually requires higher accuracy, we set the tolerance of the KKT residual Tol $= 10^{-6}$ in the experiments.

In Table \ref{table4}, we repeat the random partition for 10 times to yield the mean performances of the three algorithms. As we see in the table, the SSsNAL method not only uses less time but also needs fewer iterations than the other two algorithms under the same stopping criterion. This is due to the fast linear convergence rate of the SSsNAL method and the fact that the inner subproblem can be solved very quickly and cheaply by the semismooth Newton method, which shows again the superiorities of the SSsNAL method in solving the large-scale SVM problems.

\begin{table}[htpb]
\begin{adjustbox}{addcode={\begin{minipage}{\width}}{\caption{Comparisons of the SSsNAL, FAPG and P2GP methods for the SVR problems \eqref{Def_Dual_of_SVR_Problem} with the RBF kernel function. In the table, ``a'' denotes the SSsNAL method, ``b'' denotes the FAGP method, and ``c'' denotes the P2GP method.}\label{table4}\end{minipage}},rotate=00,center}
\centering
\small
\begin{tabular}{|cc|c|c|c|c|}
\hline
\multirow{2}{*}{\shortstack{Data \\ C:$\alpha$:$\varepsilon$}} & \multirow{2}{*}{suppvec} & \multicolumn{1}{c|}{$\text{R}_{\text{KKT}}$} & \multicolumn{1}{c|}{Time}  & \multicolumn{1}{c|}{Iter} &\multicolumn{1}{c|}{MSE}\\
& & \multicolumn{1}{c|}{a$|$b$|$c}  &  \multicolumn{1}{c|}{a$|$b$|$c}  & \multicolumn{1}{c|}{a$|$b$|$c} & \multicolumn{1}{c|}{a$|$b$|$c} \\
    \hline
    \multirow{2}{*}{\shortstack{cadata \\ 512:0.0313:0.0686}}  & \multirow{2}{*}{2}    & \multirow{2}{*}{9.8-8 $|$ 9.9-8 $|$ 4.9-7} & \multirow{2}{*}{29  $|$ 2:15 $|$ 2:07} & \multirow{2}{*}{8(3) $|$ 149  $|$ 68(340,595) }& \multirow{2}{*}{2.88-2 $|$ 2.88-2 $|$  2.88-2} \\
     &     &  &  &  &  \\ \hline
    \multirow{2}{*}{\shortstack{abalone \\ 512:0.0625:0.0559}}  & \multirow{2}{*}{2}   & \multirow{2}{*}{2.9-7 $|$ 2.8-7 $|$ 6.9-7}  & \multirow{2}{*}{2 $|$ 9 $|$ 1:11} & \multirow{2}{*}{2(9) $|$ 216 $|$ 114(635,2406)} & \multirow{2}{*}{4.63-2 $|$ 4.63-2 $|$ 4.63-2}  \\
     &     &  &  &  &  \\ \hline
    \multirow{2}{*}{\shortstack{space\_ga \\ 64:0.250:0.0630}}  & \multirow{2}{*}{4}   & \multirow{2}{*}{1.4-7 $|$ 2.2-7 $|$ 2.7-7} & \multirow{2}{*}{1 $|$ 7 $|$ 24} & \multirow{2}{*}{9(3) $|$ 661 $|$ 14(124,2603)} & \multirow{2}{*}{1.15-1 $|$ 1.16-1 $|$ 1.16-1}  \\
     &     &  &  &  &  \\ \hline
\end{tabular}
\end{adjustbox}
\end{table}

\newpage

\section{Conclusion}\label{section_5}
In this paper, we proposed a highly efficient SSsNAL method for solving the large-scale convex quadratic programming problem generated from the dual problem of the SVMs. By leveraging the primal-dual error bound result, the fast local convergence property of the AL method can be guaranteed, and by exploiting the second-order sparsity of the Jacobian when using the SsN method, the algorithm can efficiently solve the problems. Finally, numerical experiments demonstrated the high efficiency and robustness of the SSsNAL method.

\section*{Acknowledgements}\label{section_6}
We would like to thank Prof. Akiko Takeda at the University of Tokyo and Dr. Naoki Ito at Fast Retailing Co., Ltd. for sharing the codes of the FAPG method and providing some valuable comments for our manuscript. Additionally, we also thank Prof. Ingo Steinwart at University of Stuttgart for his valuable suggestions to improve the manuscript.


\end{document}